\documentclass[12pt]{article}
\usepackage[utf8]{inputenc}
\usepackage{algorithm}
\usepackage{algorithmicx}
\usepackage{algpseudocode}

\usepackage[utf8]{inputenc} 
\usepackage[T1]{fontenc}    
\usepackage[english]{babel}

\usepackage{amsfonts}
\usepackage{nicefrac}       
\usepackage{microtype}      
\usepackage{lmodern}
\usepackage{amssymb,amsmath,amsthm}
\usepackage{stmaryrd}
\usepackage{bbm,bm}
\usepackage{latexsym}
\usepackage{xcolor}
\usepackage{graphicx}
\usepackage{subfigure}
\usepackage{booktabs}
\usepackage{url}
\usepackage{hyperref}
\hypersetup{
	colorlinks=true,
	linkcolor=blue,
	filecolor=blue,
	anchorcolor=blue,
	urlcolor=cyan,
	citecolor=purple
}
\usepackage{enumerate}
\usepackage[shortlabels]{enumitem}
\usepackage{verbatim}
\usepackage{booktabs}       
\usepackage{multirow}

\usepackage{url}            
\usepackage[numbers]{natbib}

\usepackage[in]{fullpage}

\usepackage[short]{optidef}
\usepackage{algorithm}
\usepackage{thmtools}

\usepackage{tikz}
\usepackage{tcolorbox}
\usepackage[framemethod=tikz]{mdframed}

\declaretheorem{theorem}
\declaretheorem{corollary}
\declaretheorem{lemma}
\declaretheorem{proposition}

\declaretheorem{fact}

\declaretheoremstyle[qed=$\square$]{definitionwithend}
\declaretheorem{definition}

\declaretheorem[style=definitionwithend]{example}

\definecolor{gold}{rgb}{0.85,0.65,0}

\numberwithin{subsection}{section}

\let\emptyset\varnothing

\usepackage{mathrsfs}
\def\A{{\bm{A}}}
\def\B{{\mathbb{B}}}

\def\N{{\mathbb{N}}}

\def\R{{\mathbb{R}}}

\def\Z{{\mathcal{Z}}}

\def\bA{{\mathbf{A}}}

\def\bQ{{\mathbf{Q}}}

\def\cA{{\cal A}}
\def\cB{{\cal B}}

\def\cD{{\cal D}}

\def\cI{{\cal I}}
\def\cJ{{\cal J}}
\def\cK{{\cal K}}

\def\cN{{\cal N}}
\def\cO{{\cal O}}

\def\cS{{\cal S}}

\def\cU{{\cal U}}

\def\cW{{\cal W}}
\def\cX{{\cal X}}
\def\cY{{\cal Y}}

\def\sP{\mathscr{P}}
\def\sQ{\mathscr{Q}}
\def\a{{\boldsymbol{a}}}
\def\b{{\bm b}}

\def\e{{\bm e}}

\def\p{{\bm p}}
\def\q{{\bm q}}

\def\u{{\bm u}}
\def\v{{\bm v}}
\def\w{{\bm w}}
\def\x{{\bm x}}
\def\y{{\bm y}}
\def\z{{\bm z}}
\def\bz{{\bm 0}}

\def\1{{\bm 1}}

\newcommand{\rank}{\mathrm{rank} }
\newcommand{\blam}{{\boldsymbol\lambda}}

\newcommand{\bmu}{{\boldsymbol\mu}}

\newcommand{\dist}{\operatorname{dist}}

\DeclareMathOperator*{\argmin}{arg\,min}

\DeclareMathOperator{\Diag}{Diag}

\DeclareMathOperator{\inter}{int}

\DeclareMathOperator{\cl}{cl}
\DeclareMathOperator{\bd}{bd}

\DeclareMathOperator{\dom}{dom}

\usepackage{soul}

\title{{\bf \large On the Iterate Convergence of Bregman \\ Projected Gradient Method}}
\author{
	He Chen\thanks{Department of Systems Engineering and Engineering Management, The Chinese University of Hong Kong, Shatin, Hong Kong, Email: hchen@se.cuhk.edu.hk} 
	\and Anthony Man-Cho So\thanks{Department of Systems Engineering and Engineering Management, The Chinese University of Hong Kong, Shatin, Hong Kong, Email: manchoso@se.cuhk.edu.hk} 
}
\date{}
\begin{document}
	\maketitle
	\begin{abstract}
		The iterate convergence of \textit{Bregman projected gradient method} (BPGM) has remained a long-standing open problem, especially for the widely adopted Shannon entropy kernel. Existing convergence results are often limited, relying on Lipschitz continuity of the kernel's gradient or restrictive conditions on the objective function. In this paper, we develop a novel convergence analysis framework for the BPGM with the Shannon entropy kernel, yielding strong convergence results for a broad class of objective functions under linear constraints. 
		The cornerstone of our framework is a new concept called \textit{scaled Kurdyka-\L{}ojasiewicz} (SK\L{}) property, which captures the local growth behavior of a function under the Bregman geometry. We show that the SK\L{} property ensures the iterate convergence of BPGM and holds for all continuous subanalytical functions. 
		Furthermore, we prove that the BPGM sequence exhibits linear convergence if the problem possesses an SK\L\ exponent of $1/2$. We then furnish the examples of functions with the SK\L\ exponent $1/2$ by proving that the SK\L\ exponent $1/2$ is implied by the K\L{} exponent $1/2$ under strict complementarity and local Lipschitz continuity of the objective's gradient. 
		Building on these novel results, our work takes a  first step towards resolving the open problem of BPGM iterate convergence.
	\end{abstract}
	\section{Introduction}
	The emergence of large-scale machine learning and signal processing models in the early 2000s has sparked intense research on first-order methods for tackling the optimization formulations associated with these models; see, e.g., \cite{sra2012optimization,beck2017first,Jain2017nonconvex,teboulle2018simplified,so2020nonconvex} and the references therein. Consider, for
	instance, the prototypical formulation
	\begin{equation}\label{eq:model}
		\min\limits_{\x\in \cD}\ f(\x),\tag{$\sP$}
	\end{equation}
	where $f:\R^n\rightarrow \R$ is a differentiable function and $\cD\subseteq\R^n$ is a non-empty closed set. For simplicity, we assume that Problem \eqref{eq:model} has an optimal solution.  One classic strategy for tackling \eqref{eq:model} is to iteratively minimize a convex majorant of $f$ over $\cD$. In the setting where the gradient of $f$ is Lipschitz continuous on $\cD$, it is well known that $f$ can be majorized by a convex quadratic function on $\cD$. Specifically, for any $L \geq L_f$, where $L_f > 0$ is the Lipschitz constant of $\nabla f$ on $\cD$, we have
	\begin{equation}\label{eq:upper}
		f(\x)\leq f(\y) +\nabla f(\y) ^{\top} (\x-\y)+\frac{L}{2}\|\x-\y\|^2_2,\quad \forall\ \x,\y\in\cD;
	\end{equation}
	see, e.g., \cite[Theorem 2.1.5]{nesterov2018lectures}. This leads to the following iterative scheme for solving \eqref{eq:model}, where $\alpha_k>0$ is the step size in the $k$-th iteration:
	\begin{equation}\label{eq:pgd}
		\x^0\in\cD;\quad \x^{k+1}\in \argmin_{\x\in\cD}\left\{f(\x^k)+\nabla f(\x^k) ^{\top} (\x-\x^k)+\frac{1}{2\alpha_k}\|\x-\x^k\|^2_2\right\},\quad k\geq0.\tag{$\dagger$}
	\end{equation}
	Incidentally, if we denote the projection of $\u\in \R^n$ onto $\cD$ by $\Pi_{\cD}(\u)\coloneqq \argmin_{\z\in\cD}\|\u-\z\|_2^2$, then a simple calculation yields 
	$\x^{k+1} \in\Pi_{\cD}\left(\x^k-\alpha_k\nabla f(\x^k)\right) $ for all $k\geq0$, which shows that \eqref{eq:pgd} is nothing but the \textit{projected gradient method} (PGM). The convergence behavior of the PGM has been extensively studied over the years and is by now well understood. In particular, by choosing step sizes $\{\alpha_k\}_{k\geq0}$ based on the Lipschitz constant of $\nabla f$ on $\cD$, one has the following results:
	\begin{enumerate}
		\item[(P1)] If both $f$ and $\mathcal{D}$ are convex, then the sequence of function values $\{f(\x^k)\}_{k \geq 0}$ generated by the PGM converges to the optimal value of \eqref{eq:model} at a global sublinear rate. Moreover, if $f$ is non-convex and $\mathcal{D}$ is convex, then the sequence of running minima of iterate gaps $\left\{\min _{1 \leq i \leq k}\left\|\x^i-\x^{i-1}\right\|_2\right\}_{k \geq 1}$ generated by the PGM converges to zero at a global sublinear rate. See, e.g., \cite[Chapter 10]{beck2017first}. 
		\item[(P2)]  If $\cD$ is convex and \eqref{eq:model} possesses certain error bound and proper separation properties, then the sequence of iterates $\{\x^k\}_{k\geq0}$ generated by the PGM converges to a critical point of \eqref{eq:model} at a local linear rate. See, e.g., \cite[Section 3.1]{luo1993error}.
		\item[(P3)] If $f+\delta_{\mathcal{D}}$ possesses the Kurdyka-\L{}ojasiewicz (K\L{}) property, where $\delta_{\mathcal{D}}: \mathbb{R}^n \rightarrow\{0,+\infty\}$ denotes the indicator function associated with $\mathcal{D}$, then the sequence of iterates $\{\x^k\}_{k \geq 0}$ generated by the PGM converges to a critical point of  \eqref{eq:model}, provided that it is bounded. Moreover, the local rate of convergence is determined by the K\L{} exponent of $f+\delta_{\mathcal{D}}$. See, e.g., \cite[Theorem 5.3]{attouch2013convergence} and \cite{bolte2014proximal}. 
	\end{enumerate}
	As satisfying as the above advances may be, they rely crucially on the Lipschitz continuity of the gradient of $f$. In recent years, there has been much interest in designing first-order methods that can tackle \eqref{eq:model} even when $\nabla f$ may not be Lipschitz continuous on $\cD$. This is motivated in part by
	various contemporary applications, such as image processing with Poisson-type data \cite{bertero2009image,sra2009new} and computational optimal transport for data analysis \cite{peyre2019computational}. In an important development, \citet{bauschke2017descent} demonstrated how the PGM can be extended to cover a class of objective functions that do not have Lipschitz continuous gradients.  The starting point is the observation that the inequality \eqref{eq:upper} is equivalent to the convexity of the function $\x\mapsto \frac{L}{2}\|\x\|_2^2-f(\x)$ on $\cD$.  This suggests that in general, one may try to find a function $\tilde{h}$ that captures the geometry of Problem \eqref{eq:model} through the convexity of the function $\x\mapsto \tilde{h}(\x)-f(\x)$ on a certain set. To formalize this idea, \citet{bauschke2017descent} considered the setting where there exists a proper, lower semicontinuous, convex function $h:\R^n\rightarrow\R\cup\{+\infty\}$ (referred to as a \textit{kernel}) such that (i) it satisfies $\cD\subseteq\cl(\dom(h))$ and is differentiable on ${\rm int}(\dom(h))$ (in particular, the set $\cD$ is assumed to be convex and have a nonempty interior), and (ii) the function $\x\mapsto Lh(\x)-f(\x)$ is convex on ${\rm int}(\dom(h))$ for some constant $L>0$ (i.e., $f$ is $L$-relatively smooth with respect to $h$ on ${\rm int}(\dom(h))$). In such a setting,
	it is easy to derive an iterative scheme similar to the PGM for solving Problem \eqref{eq:model}. Indeed, condition (ii) above is equivalent to the inequality
	\[f(\x) \leq f(\y)+\nabla f(\y)^{\top}(\x-\y)+L\left(h(\x)-h(\y)-\nabla h(\y)^{\top}(\x-\y)\right), ~ \forall\ \x, \y \in {\rm int}(\dom(h)),\]
	which provides a convex majorant of $f$ on ${\rm int}(\dom(h))$. Thus, similar to the derivation of the PGM discussed earlier, one can define the following iterative scheme for solving \eqref{eq:model}, where $\alpha_k>0$ is
	the step size in the $k$-th iteration and $D_h(\x, \y) \coloneqq h(\x)-h(\y) -\nabla h(\y)^{\top}(\x-\y)$ is the Bregman
	distance associated with the kernel $h$; see \cite{bauschke2017descent} and cf. \cite{tseng2010approximation,lu2018relatively,takahashi2025approximate}:
	\[\x^0\in{\rm int}(\dom(h)); \]
	\begin{equation}\label{eq:iterbreg}
		\x^{k+1} \in\argmin_{\x\in\cD}\left\{f(\x^k)+\nabla f(\x^k)^{\top}(\x-\x^k)+\frac{1}{\alpha_k} D_h(\x, \x^k)\right\}, \quad \forall\ k \geq 0.\tag{$\ddagger$}
	\end{equation}
	Hereafter, we shall refer to \eqref{eq:iterbreg} as the \textit{Bregman projected gradient method} (BPGM). It is worth noting that even when $\nabla f$ is Lipschitz continuous on $\cD$, it could be more advantageous to use
	the BPGM with a suitably chosen kernel $h$ than to use the PGM (which is an instance of the BPGM with kernel $h(\x)=\frac12\|\x\|_2^2$) to tackle Problem \eqref{eq:model}, especially when the update \eqref{eq:iterbreg} is computationally simpler than the update \eqref{eq:pgd}.
	
	The unified treatment of the PGM and BPGM just outlined provides a rather transparent way to develop sublinear convergence results similar to those in (P1) above for the BPGM with various kernels; see \cite[Theorem 4.1 and Theorem 5.1]{teboulle2018simplified} and cf. \cite[Theorem 3.1] {lu2018relatively} and \cite[Proposition 4.1]{bolte2018first}. However, in the settings of (P2) and (P3) above, where Problem \eqref{eq:model} may have a non-convex objective function but possesses additional (and by now standard) regularity properties, it remains unclear how the said unified treatment can facilitate the development of iterate convergence results for the BPGM that hold for a sufficiently general class of kernels.  Consequently, the question of iterate convergence of BPGM  has remained open for a long time, {  with relatively few results obtained.  In the 1990s, \citet{iusem1999central} showed an iterate convergence result for the linear programming setting, but the general case remained unresolved. Only recently did \citet{bauschke2017descent} prove iterate convergence of BPGM for convex objective functions. 
	For nonconvex objective functions,} most existing results concerning the convergence (rate) of the BPGM iterates either (i) require the kernel $h$ to satisfy $\dom(h)=\R^n$ and have a Lipschitz continuous gradient on any bounded subset of $\R^n$ (e.g., \cite{teboulle2018simplified,wu2021inertial,zhu2021level,latafat2022bregman,takahashi2025approximate,takahashi2022new}), which excludes many entropy-type kernels that feature prominently in applications; or
	(ii) assume that the objective function satisfies certain strong convexity-type condition with respect to the kernel (e.g., \cite{bauschke2019linear,azizian2022rate}), which can be restrictive in practice; or (iii) impose certain Bregman growth condition on the objective function (e.g., \cite{bauschke2019linear,zhang2021proximal}), which in general does not seem to be easy to verify. Part of the difficulty in using the error bound property in (P2) or the KŁ property in (P3) to establish the iterate convergence of the BPGM iterates for a general class of kernels is that the geometry captured by those properties is induced by the Euclidean distance and may not be compatible with the geometry induced by a non-Euclidean Bregman distance. In fact, as detailed in \cite{chen2024spurious}, the latter geometry inherently induces spurious stationary points when the kernel has a non-Lipschitz gradient, making it difficult to determine whether the accumulation points of the BPGM iterates are genuine stationary points or spurious ones. These difficulties call for a new analysis framework for establishing the iterate convergence of BPGM. 
	
	As a step towards resolving the open problem of BPGM iterate convergence, in this paper, we focus on the setting where Problem \eqref{eq:model} satisfies the following assumptions:
	\begin{enumerate}
		\item[(A1)] The feasible set $\mathcal{D}$ is a polyhedron of the form $\mathcal{D}=\mathcal{D}_P:=\left\{\x \in \mathbb{R}^n: \A \x=\b, \x\geq \bz\right\}$, where $\A \in \R^{m \times n}$ and $\b \in \R^m$ are given with $\A$ being of full row rank. Furthermore, we have $\cD_P \cap \R_{++}^n \neq \emptyset$. 
		\item[(A2)] The objective function $f$ is continuously differentiable and $L$-relatively smooth with respect to the Shannon entropy kernel $\x\mapsto h_S(\x)\coloneqq \sum_{i=1}^nx_i\log(x_i)$ on the set $\cD_P$ but not necessarily convex.
	\end{enumerate}
	Such a setting appears in many applications, and a popular approach to tackling it is to apply the BPGM; see, e.g., \cite{xu2019gromov,liconvergent}. Currently, the only result that establishes the convergence rate of the BPGM iterates for the above setting appears in \cite[Section 5]{azizian2022rate}. However, it assumes, among other things, that the critical point $\bar{\x}$ to which the iterates converge is locally unique (i.e., there is no other critical point in a neighborhood of $\bar{\x}$) and that the initial iterate lies sufficiently close to $\bar{\x}$. To obtain
	less restrictive convergence results, we note that the Bregman distance $D_{h_S}$
	becomes unbounded as its second argument approaches the boundary of $\dom(h_S)=\R^n_+$. This suggests that the usual \textit{sufficient decrease} and \textit{relative error} conditions, which are formulated in terms of the Euclidean distance and play a central role in error bound/K\L{}-based convergence analyses of iterative methods (see, e.g., \cite{luo1993error,attouch2013convergence}), may not be suitable for characterizing the behavior of the BPGM iterates. We circumvent this difficulty by observing that the geometry induced by the Euclidean distance can be made compatible with that induced by the Bregman distance $D_{h_S}$ through a suitable scaling transformation. Specifically, we show that the BPGM iterates satisfy certain \textit{scaled versions} of the sufficient decrease and relative error conditions. Such a discovery motivates us to introduce a novel variant of the K\L{} property, which we call the \textit{scaled K\L} (SK\L{}) property, to capture the geometry induced by the { Bregman} distance $D_{h_S}$.  In general, the SK\L{} property facilitates the convergence analysis of iterative methods that satisfy the scaled sufficient decrease and relative error conditions, in the same way that the K\L{} property facilitates the convergence analysis of iterative methods that satisfy the usual sufficient decrease and relative error conditions. {  The associated scaled norm uses the kernel Hessian as the scaling matrix and can be interpreted as a Hessian-Riemannian norm (see, e.g., \cite[Sec. 2.2]{alvarez2004hessian}) with a barrier operator (see, e.g., \cite[Sec. 1]{bolte2003barrier}). However, unlike the K\L{} property under Hessian-Riemannian geometry (see, e.g., \cite[Definition 3]{huang2022riemannian}), which is restricted to the interior region $\inter(\dom(h_S))=\R^n_{++}$, our SK\L{} property applies to the entire kernel domain $\dom(h_S)=\R^n_+$, including boundary points. } In particular, when applying the BPGM to Problem \eqref{eq:model} in which assumptions (A1) and (A2) hold and the function $f+\delta_{\cD}$ possesses the SK\L{} property, we establish global convergence of the iterates to a critical point of
	the problem as long as they are bounded.  If in addition the SK\L{} exponent of the function $f + \delta_{\cD}$ is $1/2$, then we have \textit{local linear convergence} of the iterates.
	
	Now, a fundamental issue is whether there are any functions possessing our newly introduced SK\L{} property. As it turns out, the SK\L{} property has a close connection with the classic K\L{} property. Indeed, recall that any continuous subanalytic function possesses the K\L{} property \cite[Theorem 3.1]{bolte2007lojasiewicz}. We show that the same is true for the SK\L{} property. In the context of Problem \eqref{eq:model}, when the feasible set $\cD$ satisfies (A1) and the objective function $f$ satisfies (A2), we see that
	$f + \delta_{\cD}$ possesses the SK\L{} property when $f$ is subanalytic. This yields a large class of functions for which global convergence of the BPGM can be established. 
	Furthermore, under some standard regularity assumptions { and local Lipschitz continuity of $\nabla f$}\footnote{{   Although the Lipschitz continuity of $\nabla f$ may be stronger than the relative smoothness of $f$, it is typically easy to verify in practice. Moreover, this condition is necessary for establishing the exponent result, as it is compatible with the K\L{} exponent $1/2$, which is defined within the Euclidean geometry.  }}, we show that the function $f + \delta_{\cD}$ has an SK\L{} exponent of $1/2$ whenever
	it has a K\L{} exponent of $1/2$. This allows us to take advantage of existing K\L{} exponent results (e.g., \cite{luo1993error,zhou2017unified,li2018calculus}) to establish linear convergence of the BPGM. It is worth noting that our convergence results for the BPGM are the first in the literature that rely essentially only on the K\L{} property (or subanalyticity) of the problem at hand and do not require any assumptions that are non-standard and/or hard to verify. Building on these { novel} results, our work takes a {  first} step toward resolving the open problem of BPGM iterate convergence.
	
	\textbf{Organization.} The rest of this paper is organized as follows. Sec.~\ref{sec:pre} gives the basic setup and some properties of the BPGM. Sec.~\ref{sec:wkl} introduces the SK\L\ property and proves that the SK\L\ property (i) ensures the iterate convergence of BPGM and (ii) holds for continuous subanalytic functions. In Sec.~\ref{sec:linear}, we develop the linear convergence of the BPGM under the SK\L{} exponent $1/2$ and establish the relationship between the SK\L{} exponent $1/2$ and K\L{} exponent $1/2$. Finally, we give some closing remarks in Sec.~\ref{sec:end}.
    
	\textbf{Notation.}
	The notation used in this paper is mostly standard. We use $[n]$ to denote the set $\{1,2,\ldots,n\}$ for each positive integer $n$.  For an index set $\cJ\subseteq[n]$, we denote its cardinality by $|\cJ|$. Denote $\R\cup\{+\infty\}$ by $\overline{\R}$.  We use $\1_n$ to denote the $n$-dimensional all-one vector. For a vector $\x\in\R^n$, we denote its $i$-th element (resp. subvecter indexed by $\cJ$) by $x_i$ (resp. $\x_{\cJ}$).
	We use
	$\dist(\x,\Z)\coloneqq \inf_{\z\in\Z}\|\x-\z\|_2$ 
	to denote the distance of $\x$ to a set $\Z\subseteq\R^n$ and $\B(\x,r)$ to denote the ball $\{\y\in\R^n:\|\y-\x\|_2\leq r\}$. For two vectors $\x,\y\in\R^n$, we use $\x\leq\y$ to denote the element-wise inequalities $x_i\leq y_i,i\in[n]$. For any univariate function $g:\cU\rightarrow\R$ with $\cU\subseteq \R$ and any vector $\u\in\R^n$, we define $g(\u)\coloneqq(g(u_1),\ldots,g(u_n))$ and $\Diag(\u)$ to be the diagonal matrix whose diagonal elements are given by the entries of $\u$. 
	
	We further adopt some concepts from variational analysis; see \cite{rockafellar2009variational} for details. Let $g: \mathbb{R}^n \rightarrow \overline{\R}$ be a proper lower semicontinuous function. The domain of $g$ is defined by $\operatorname{dom}(g)=\left\{\boldsymbol{x} \in \mathbb{R}^n: g(\boldsymbol{x})<+\infty\right\}$. The Fr${\rm \acute{e}}$chet subdifferential of $g:\R^n\rightarrow \overline{\R}$ at $\bm{x}\in {\rm dom}(g)$, denoted by $\widehat{\partial} g(\bm{x})$, is the set of vectors $\v\in\mathbb{R}^{n}$ such that \begin{displaymath}\lim\limits_{\bm{y}\neq \bm{x}}\inf\limits_{\bm{y}\rightarrow \bm{x}}\frac{g(\bm{y})-g(\bm{x})-\langle \bm{v}, \bm{y}-\bm{x}\rangle}{\|\bm{y}-\bm{x}\|_2}\ge0.\end{displaymath}   
	We write $\x^{k}\underset{g}{\rightarrow}\x$ if $\x^k\to\x$ and $g(\x^k)\to g(\x)$. The limiting subdifferential of $g$ at $\bm{x}\in {\rm dom} (g)$, denoted by $\partial g(\bm{x})$, is defined by $\partial g(\bm{x})\coloneqq\{\bm{v}\in\mathbb{R}^n: \exists\ \bm{x}^k\underset{g}{\rightarrow} \bm{x},\ \bm{v}^k\in\widehat{\partial} g(\bm{x}^k)\rightarrow \bm{v} \}.$
	
	For a set-valued mapping $S:\R^n\rightrightarrows\R^n$, its domain is defined by $\dom(S)\coloneqq\{\x\in\R^n:S(\x)\neq\emptyset\}$. The outer limit of $S$ at ${\x}\in\dom(S)$ is defined by
	\[\limsup\limits_{\y\to{\x}}S(\y)\coloneqq \left\{\v\in\R^n:\exists\ \y^k\to{\x},\ \exists\ \v^k\to\v\text{ with }\v^k\in S(\y^k)\right\}.\]
	\section{Preliminaries}\label{sec:pre}
	Let us begin with some basic preparations. Under assumptions (A1) and (A2), we write our problem of interest more explicitly as
	\begin{equation}\label{eq:obj}
		\begin{array}{cl}
			\min & f(\boldsymbol{x}) \\
			\text { subject to } & \boldsymbol{x} \in \mathcal{D}_P=\left\{\boldsymbol{x} \in \mathbb{R}^n: \boldsymbol{A} \boldsymbol{x}=\boldsymbol{b}, \boldsymbol{x} \geq \mathbf{0}\right\}. \tag{$\sQ$}
		\end{array}
	\end{equation}
	As usual, we say that $\bar\x\in\R^n$ is a critical point of \eqref{eq:obj} if it satisfies the first-order optimality conditions associated with \eqref{eq:obj}; i.e., (i) $\bar\x\in\cD_P$ and (ii) there exist multipliers $\bar{\bmu}\in\R^m$ and $\bar{\blam}\in\R^n$ satisfying
	\begin{equation}\label{kkt-condition}
		\nabla f(\bar\x)+\A^{\top}\bar\bmu-\bar\blam=\bz,\quad  \bar\blam\geq\bz\quad \text{and}\quad {\bar\blam}^{\top}\bar\x=0.
	\end{equation}
	Let $F\coloneqq f+\delta_{\cD_P}$.
	It can be easily shown that $\bar\x$ is a critical point of \eqref{eq:obj} in the sense defined above if and only if $\bz\in \partial F(\bar\x)$, and $\partial F$ has a explicit expression:
	\begin{equation}\label{eq:equi-con}
		\partial F(\x)=\left\{\nabla f(\x)+\A^{\top}\bmu-\blam:\blam\geq\bz,\ \blam^{\top}\x=0\right\},\quad \forall\ \x\in\cD_P. 
	\end{equation}
	Now, assumption (A2) suggests that one can apply the BPGM with kernel $h = h_S$ to tackle \eqref{eq:obj}. The following alternative form of the BPGM update \eqref{eq:iterbreg} with $h=h_S$ and $\cD=\cD_P$ will prove useful for our subsequent development. 
	\begin{proposition}\label{pro:explicit}
		For any $k>0$, given $\x^k\in\cD_P\cap\R^n_{++}$, the next BPGM iterate with kernel $h=h_S$ and step size $\alpha_k>0$ is
		\begin{equation}\label{eq:explicitbreg}
			\x^{k+1}= \Diag\left(\exp\left(-\alpha_k \left(\nabla f(\x^k)+\A^{\top}\bmu^k\right) \right)\right)\x^k\in\cD_P\cap\R^n_{++}, \tag{$\ddagger\ddagger$}
		\end{equation}
		where 
		\begin{equation*}\label{eq:V}
			\bmu^k=\argmin\limits_{\bmu\in\R^m}\ \left\{\frac1{\alpha_k}\exp\left(-\alpha_k\left(\nabla f(\x^k)+\A^{\top}\bmu\right)\right)^{\top}\x^k+\b^{\top}\bmu\right\}.
		\end{equation*}
	\end{proposition}
	\begin{proof}
		We first verify the well-definedness of $\x^{k+1}$ in \eqref{eq:iterbreg} for $h=h_S$. By \cite[Theorem 1, p. 798]{juditsky2023unifying} (see, also, \cite[Lemma 2.3]{teboulle2018simplified} and \cite[Lemma 2.1]{bauschke2019linear}), the update \eqref{eq:iterbreg} admits a unique solution $\x^{k+1}$, which lies in $\cD_P\cap\R^n_{++}$ when $h=h_S$.
		
		Then, the first-order optimality condition of \eqref{eq:iterbreg} yields
		\begin{equation*}
			-\nabla f(\x^k)-\frac1{\alpha_k}\left(\log\left({\x^{k+1}}\right)-\log\left({\x^k}\right)\right)\in\partial \delta_{\cD_P}\left(\x^{k+1}\right).
		\end{equation*}   
		Note that we have $\partial \delta_{\cD_P}(\x^{k+1})=\{\A^{\top}\bmu:\bmu\in\R^m\}$ by \eqref{eq:equi-con} and $\x^{k+1}>\bz$. The above optimality condition is equivalent to
		\begin{equation*} 
			-\nabla f\left(\x^k\right)-\frac1{\alpha_k}\left(\log\left({\x^{k+1}}\right)-\log\left({\x^k}\right)\right)=\A^{\top}\bmu^k \text{ for some }\bmu^k\in\R^m,
		\end{equation*} 
		and can be further equivalently formulated as
		\begin{equation} \label{eq:kkt}
			\x^{k+1}= \Diag\left(\exp\left(-\alpha_k \left(\nabla f(\x^k)+\A^{\top}\bmu^k\right) \right)\right) \x^k  \text{ for some }\bmu^k\in\R^m.
		\end{equation}
		It remains to show that $\bmu^k$ in \eqref{eq:kkt} is nothing but $\argmin_{\bmu\in\R^m}V(\bmu)$, where
		\[V(\bmu)\coloneqq \frac1{\alpha_k}\exp\left(-\alpha_k\left(\nabla f(\x^k)+\A^{\top}\bmu\right)\right)^{\top}\x^k+\b^{\top}\bmu.\]
		First, 	we show that the function $V$ is strongly convex. It suffices to establish the strong convexity of the function $\bmu\mapsto V(\bmu)-\b^{\top}\bmu$, which follows from (i) $\bmu\mapsto V(\bmu)-\b^{\top}\bmu$ is a composition of the strongly convex function $\y\mapsto(1/\alpha_k)\exp\left(-\alpha_k\y\right)^{\top}\x^k$ and the linear transformation $\y=\nabla f(\x^k)+\A^{\top}\bmu$; and (ii) $\A^{\top}$ is of full column rank by assumption (A1).
		
		Then, we prove that the vector $\bmu^k$ in \eqref{eq:kkt} minimizes $V$. Due to the strong convexity of $V$, it suffices to check that $\bmu^k$ satisfies $\nabla V(\bmu^k)=\bz$, which can be explicitly expressed as
		\[-\A\Diag\left(\exp\left(-\alpha_k \left(\nabla f(\x^k)+\A^{\top}\bmu^k\right) \right)\right)\x^k +\b=\bz. \]
		Note that $\A\x^{k+1}=\b$ by $\x^{k+1}\in\cD_P$. The above equality is directly ensured by substituting \eqref{eq:kkt} into $\A\x^{k+1}=\b$. We conclude that $\bmu^k\in\argmin_{\bmu\in\R^m}V(\bmu)$.
		
		Combining the established facts that $\bmu^k\in\argmin_{\bmu\in\R^m}V(\bmu)$ and $V$ is strongly convex, we see that $\bmu^k=\argmin_{\bmu\in\R^m}V(\bmu)$. This completes the proof.
	\end{proof}
	To study the convergence behavior of the BPGM iterates $\{\x^k\}_{k\geq0}$ generated according to \eqref{eq:explicitbreg}, one can try to adopt the K\L{}-based convergence analysis framework developed in \cite{attouch2013convergence}. A key step in the framework is to show that the iterates $\{\x^k\}_{k\geq0}$ generated by the BPGM satisfy the following two conditions:
	\begin{itemize}
		\item \textbf{Sufficient Decrease.} There exists a constant $\theta_1>0$ such that for all $k \geq 0$, we have $F(\x^{k+1})+\theta_1\|\x^{k+1}-\x^k\|^2_2\leq F(\x^k)$.
		\item \textbf{Relative Error.} There exists a constant $\theta_2>0$ such that for all $k \geq 0$, $\|\w^{k+1}\|_2\leq \theta_2\cdot \|\x^{k+1}-\x^k\|_2$ for some $\w^{k+1}\in\partial F(\x^{k+1})$.
	\end{itemize}
	However, both of these conditions involve the iterate gap $\|\x^{k+1}-\x^k\|_2$, which may not be suitable for measuring the progress of the BPGM iterates generated according to \eqref{eq:explicitbreg} due to the multiplicative nature of the update. Indeed, one can construct instances of Problem \eqref{eq:obj} for which the BPGM iterates fail to satisfy the relative error condition. 
	\begin{example}[Failure of Relative Error Condition for BPGM]\label{example:lp}
		Consider problem 
		\begin{equation*}
			\begin{array}{cl}
				\min & f(x_1,x_2)=-x_1 \\
				\text { subject to } & (x_1,x_2)\in \mathcal{D}_P=\left\{\boldsymbol{x}=(x_1,x_2) \in \mathbb{R}^2: x_1+x_2=1,x_1,x_2\geq0 \right\}. 
			\end{array}
		\end{equation*}
		We set the initial point as $\x^0=(1/2,1/2)$ and the step size $\alpha_k=1$ for all $k\geq0$. Since the objective function $f$ is convex, the BPGM is guaranteed to find the unique optimal solution $(1,0)$ by \cite[Theorem 2]{bauschke2017descent}; i.e., $\lim_{k\to\infty}\x^k=(1,0)$.
		It follows that $\|\x^{k+1}-\x^k\|_2\rightarrow0$.
		On the other hand, by \eqref{eq:equi-con} and  $\x^k\in\R^2_{++}$, we have  the subdifferential expression $\partial F(\x^k)=\{(-1,0)+(1,1)\mu^k:\mu^k\in\R\}$, which yields $\dist(\bz,\partial F(\x^k))={\sqrt{2}}/{2}$ for all $k\geq0$. It follows that all vectors $\w^k\in\partial F(\x^k)$ satisfy $\|\w^k\|_2\geq{\sqrt{2}}/{2}$. Therefore, there is no scalar $b>0$ such that for all $k\geq0$, 
		\[\left\|\w^{k+1}\right\|_2\leq b\left\|\x^{k+1}-\x^k\right\|_2\text{ for some } \w^{k+1}\in \partial F\left(\x^{k+1}\right).\]
		We conclude that the relative error fails for the BPGM in this simple problem.
	\end{example}
	Due to the absence of the relative error condition,
	one cannot deduce $\dist(\bz,\partial F(\x^k))\to0$ from $\|\x^k-\x^{k+1}\|_2\to0$. Consequently, it remains unclear whether the limiting points of BPGM iterates are critical for general objective functions; see \cite{chen2024spurious} for an in-depth discussion on this issue. In what follows, we develop a useful proposition that ensures the limiting point of BPGM iterates is critical, provided the BPGM iterates converge. To proceed, we first specify the setting for the BPGM sequence:
	\begin{enumerate}
		\item[(A3)] The BPGM sequence $\{\x^k\}_{k\geq0}$ is generated according to \eqref{eq:explicitbreg} with $\x^0\in\cD_P\cap\R^n_{++}$ and $\alpha_k\in[\underline{\alpha},\bar{\alpha}]$ for all $k\geq0$, where $0<\underline{\alpha}<\bar{\alpha}<1/L$.
	\end{enumerate}
	\begin{proposition}\label{pro:escape}
		Suppose that (A1)---(A3) hold and that the BPGM sequence $\{\x^k\}_{k\geq0}$ converges to $\bar\x$. Then, $\bar\x$ is a critical point. 
	\end{proposition}
	The proof of Proposition \ref{pro:escape} relies on the following lemma, which plays a crucial role in characterizing the behavior of BPGM iterates. 
	\begin{lemma}\label{le:omit}
		Suppose that (A1) and (A2) hold. Then, for every $\x^*\in\cD_P$, there exist scalars $\beta,\rho>0$ such that for all $\alpha>0$ and $\x\in\cD_P\cap\R^n_{++}$ with $\|\x-\x^*\|_2\leq\rho$, it holds that
		\[\left\|\nabla f(\x)+\A^{\top}\bmu\right\|_2\leq \beta,\]
		where $\bmu=\argmin\limits_{\bmu^{\prime}\in\R^m}\ \{\frac1{\alpha}\exp\left(-\alpha\left(\nabla f(\x)+\A^{\top}\bmu^{\prime}\right)\right)^{\top}\x+\b^{\top}\bmu^{\prime}\}$.
	\end{lemma}
	
	\section{Convergence Analysis Based on SK\L{} Property}\label{sec:wkl}
	As we discussed in Sec.~\ref{sec:pre}, the classic sufficient decrease and relative error properties are not suitable for studying the BPGM sequence with $h=h_S$. This motivates us to develop a new characterization for the BPGM sequence. As it turns out, the behavior of the BPGM sequence with $h=h_S$ can typically be characterized by the following scaled sufficient decrease and scaled relative error conditions:
	\begin{itemize}
		\item \textbf{Scaled Sufficient Decrease.} There exists a constant $\theta_1>0$ such that for all $k\geq0$, we have $F(\x^{k+1})+\theta_1\|\Diag(\sqrt{\x^k})^{-1}(\x^{k+1}-\x^k)\|^2_2\leq F(\x^k)$.
		\item \textbf{Scaled Relative Error.} There exists a constant $\theta_2>0$ such that for all $k\geq0$, we have $\|\Diag(\sqrt{\x^k})\w^{k}\|_2\leq \theta_2\cdot \|\Diag(\sqrt{\x^k})^{-1}(\x^{k+1}-\x^k)\|_2$ for some $\w^{k}\in\partial F(\x^{k})$.
	\end{itemize}
    { We note that similar scaling ideas are ubiquitous in optimization theory. For example, \citet{nesterov1997self} employed scaled norms in the convergence analysis of interior-point methods.} To derive a convergence mechanism from the above scaled conditions, we revisit the arguments used to establish the abstract convergence result in \cite{attouch2013convergence}, where the K\L\ property is combined with   sufficient decrease and relative error to guarantee iterate convergence. Since the scaled sufficient decrease and scaled relative error conditions are multiplicative in nature, it is natural to seek a corresponding multiplicative analogue of the K\L{} property. This motivates the following definition:
	\begin{definition}[SK\L{} Property]\label{def:skl}
		The function $H:\mathbb{R}^n_+ \rightarrow \overline{\R}$ is said to have the {\rm  SK\L} property at $\x^* \in \operatorname{dom} (\partial H)$ if there exists a scalar $\eta \in(0,+\infty]$, a neighborhood $\cU$ of $\x^*$, and a continuous concave function $\varphi:[0, \eta) \rightarrow \mathbb{R}_{+}$ with $\varphi(0)=0$ such that
		\begin{enumerate}[label={{\rm (\roman*)}}]
			\item $\varphi$ is continuously differentiable on $(0, \eta)$ with $\varphi^{\prime}>0$ over $(0, \eta)$,
			\item for all $\x\in\cU$ with $H(\x^*)<H(\x)<H(\x^*)+\eta$, it holds that
			$$
			\varphi^{\prime}\left(H(\x)-H(\x^*)\right) \cdot\operatorname{dist}\left(\bz, \Diag\left(\sqrt{\x}\right)\partial H(\x)\right) \geq 1.
			$$
		\end{enumerate}		
		If $\varphi(s)= \tilde{c} s^{1-\alpha}$ for some $\tilde{c}>0$ and $\alpha \in[0,1)$, then we say that $H$ has the SK\L\ property at ${\x}^*$ with an exponent of $\alpha$. If $H$ possesses the SK\L\ property at every $\x^*\in\dom(\partial H)$, then we call $H$ an SK\L\ function. Further, if  $H$ is an SK\L{} function and has exponent $\alpha$ at every critical point, then we call $H$ an SK\L\ function with an exponent of $\alpha$.
	\end{definition}
The above scaled properties form a new analysis framework\footnote{ {  Note that $\|\Diag(\sqrt{\x^k})^{-1}(\x^{k+1}-\x^k)\|_2$ and $\dist(\Diag(\sqrt{\x})\partial H(\x))$ can be interpreted, respectively, as the Hessian-Riemannian norm $\|\x^{k+1}-\x^k\|_{\nabla^2h_S(\x^k)}$ and minimal dual norm $\min_{\w\in \partial H(\x)}\|\w\|_{\nabla^2h_S(\x)^{-1}}$. Nevertheless, unlike the K\L{} framework under Hessian-Riemannian geometry, which is restricted on the interior region $\R^n_{++}$, our framework, particularly the SK\L{} property, applies to the entire kernel domain $\R^n_+$ and facilitates convergence analysis to boundary points $\x^*\in\bd(\R^n_+)$. } } specialized for the BPGM with $h=h_S$, based on which we establish the convergence of the BPGM sequence.
\begin{theorem}[Convergence under SK{\L} Property]\label{th:convergence}
	Suppose that (A1)---(A3) hold.
	Suppose further that $F=f+\delta_{\cD_P}$ is an SK\L\ function and  the BPGM sequence $\{\x^k\}_{k\geq0}$ is bounded. Then, the sequence $\{\x^k\}_{k\geq0}$ converges to a critical point of $F$.
\end{theorem}
The remaining task is to determine when the newly introduced SK\L\ property holds for the function $F$. As it turns out, every continuous subanalytic function is an SK\L\ function. 
\begin{proposition}\label{pro:skl}
	Every proper, subanalytic function $H: \R^n_+\to\overline{\R}$ that is continuous on $\dom(H)$ is an SK\L{} function.
\end{proposition}
Theorem \ref{th:convergence} and Proposition \ref{pro:skl} are significant, as they ensure the iterate convergence of BPGM for a broad class of objective functions, demonstrating the power of our new analysis framework and allowing us to take a major step towards resolving the problem of BPGM iterate convergence. These results are in sharp contrast to existing ones that assume restrictive conditions on the objective function or require the Lipschitz continuity of the gradient of the kernel.

\subsection{Proof of Theorem \ref{th:convergence}}
For the proof of Theorem \ref{th:convergence}, we proceed as follows. {  To begin, we establish} a descent lemma for the BPGM under relative smoothness, following \cite[Lemma 1]{bauschke2017descent}. Then, we show that the BPGM sequence satisfies the scaled sufficient decrease and relative error properties. Combining these properties and the SK\L{} inequality, we establish a local convergence result for the BPGM by invoking the methodology in \cite[Lemma 2.6]{attouch2013convergence}. Then, Theorem \ref{th:convergence} directly follows.
\subsubsection{Properties of BPGM sequence}
\begin{lemma}[Decrease Property of BPGM]\label{le:suffDescent}
	Suppose that (A1)---(A3) hold. Then, the BPGM sequence $\{\x^k\}_{k\geq0}$ satisfies
	\begin{equation}\label{eq:BPGe_suffDescent}
		F(\x^{k+1})-F(\x^k)\leq -  \left(\frac1{\bar{\alpha}}-{L}\right) D_{h_S}\left(\x^{k+1},\x^k\right),\qquad\forall~k\geq0.
	\end{equation} 
\end{lemma}
\begin{proof}
Since $f$ is $L$-relatively smooth on $\cD_P$, by definition, the function $L h_S-f$ is convex on $\cD_P$, which yields
\[L h_S(\x^{k+1})-f(\x^{k+1})-\left(L h_S(\x^k)-f(\x^k)\right)\geq\left(L\nabla h_S(\x^k)-\nabla f(\x^k)\right)^{\top}\left(\x^{k+1}-\x^k\right).\]
Rearranging the above inequality and using the definition of $D_{h_S}(\x^{k+1},\x^k)$, we have
\begin{equation}\label{eq:rl}
	f(\x^{k+1})-f(\x^k)\leq\nabla f(\x^k)^{\top}\left(\x^{k+1}-\x^k\right)+L D_{h_S}\left(\x^{k+1},\x^k\right).
\end{equation}
On the other hand, the optimality of $\x^{k+1}$ for the subproblem \eqref{eq:iterbreg} implies
\[\nabla f(\x^k)^{\top}\left(\x^{k+1}-\x^k\right)+\frac{1}{\alpha_k}D_{h_S}\left(\x^{k+1},\x^k\right)\leq \nabla f(\x^k)^{\top}\left(\x^k-\x^{k}\right)+\frac{1}{\alpha_k}D_{h_S}\left(\x^{k},\x^k\right)=0.\]
It follows that 
\[\nabla f(\x^k)^{\top}\left(\x^{k+1}-\x^k\right)\leq-\frac{1}{\alpha_k}D_{h_S}\left(\x^{k+1},\x^k\right)\leq-\frac{1}{\bar{\alpha}}D_{h_S}\left(\x^{k+1},\x^k\right). \] 
This, together with \eqref{eq:rl} and the fact that $F=f$ on $\cD_P$, implies
\[F(\x^{k+1})-F(\x^k)\leq -\left(\frac1{\bar{\alpha}}-L\right)D_{h_S}\left(\x^{k+1},\x^k\right).\]
We complete the proof.
\end{proof}
\begin{proposition}[Scaled Conditions]\label{pro:relative_error} Suppose that (A1)---(A3) hold. Consider an arbitrary $\x^*\in\cD_P$, let $\beta,\rho > 0$ be the scalars in Lemma \ref{le:omit} associated with $\x^*$, and define the index set $\cK\coloneqq\{k:\|\x^k-\x^*\|_2\leq\rho\}$. Then, the following hold:
	\begin{enumerate}[label={{\rm (\roman*)}}] 
		\item{\rm \textbf{(Bounded Iterative Ratio).}}  For all $k\in\cK$, we have \begin{equation}\label{eq:beta}
			\x^{k+1}\leq \exp\left(\bar{\alpha}\beta\right)\cdot\x^k.
		\end{equation}
		\item \rm{\textbf{(Scaled Sufficient Decrease).}} Let $\kappa_1 = (1/ {\bar{\alpha}}-{L})\exp(-\bar{\alpha}\beta)$. For all $k\in\cK$, we have
		\begin{equation}\label{eq:relative_error1}
			F(\x^{k+1})-F(\x^k)\leq -\frac{\kappa_1}2\left\|\Diag\left(\sqrt{\x^k}\right)^{-1} \left({\x^{k+1}-\x^k}\right) \right\|^2_2.
		\end{equation}
		\item \rm{\textbf{(Scaled Relative Error).}} Let $\kappa_2= \underline{\alpha}\exp(-\underline{\alpha}\beta)$. For all $k\in\cK$, we have
		\begin{equation}\label{eq:relative_error2}
			\dist\left(\bz, \Diag\left(\sqrt{\x^k}\right)\partial F\left(\x^k\right)\right)\leq\frac1{\kappa_2}\left\|\Diag\left(\sqrt{\x^k}\right)^{-1} \left({\x^{k+1}-\x^k}\right) \right\|_2.
		\end{equation}
	\end{enumerate}
\end{proposition}
\begin{proof}
	(i) By Proposition \ref{pro:explicit}, there exists a vector $\bmu^k\in\R^m$ such that 
	\begin{equation}\label{eq:updatek}
		\x^{k+1}=\Diag\left(\exp\left(-\alpha_k\left(\nabla f(\x^k)+\A^{\top}\bmu^k\right)\right)\right)\x^k\leq \exp\left(\alpha_k\left\|\nabla f(\x^k)+\A^{\top}\bmu^k\right\|_2\right)\x^k.
	\end{equation}
	Note that $\beta,\rho > 0$ are taken from Lemma \ref{le:omit} and $\|\x^k-\x^*\|_2\leq\rho$ for $k\in\cK$. Lemma \ref{le:omit} gives \begin{equation*}\label{eq:betak}
		\left\|\nabla f(\x^k)+\A^{\top}\bmu^k\right\|_2\leq\beta,\qquad \forall\ k\in\cK. 
	\end{equation*}
	Combining this upper bound with \eqref{eq:updatek} and using $\alpha_k\leq\bar{\alpha}$, we obtain \eqref{eq:beta}.
	
	(ii)  Let $D_{h_S}(x,y)\coloneqq x\log(x)-y\log(y)-(1+\log(y))(x-y)$ for $x\in\R_+$, $y\in\R_{++}$. Observe that the decrease value $D_{h_S}(\x^{k+1},\x^k)$ in \eqref{eq:BPGe_suffDescent} can be written as
	\[D_{h_S}\left(\x^{k+1},\x^k\right)=\sum_{i=1}^nD_{h_S}\left(x^{k+1}_i,x^k_i\right).\]
	Hence, Lemma \ref{le:suffDescent} implies that
	\begin{equation}\label{eq:estimateD0}
		F(\x^{k+1})-F(\x^k)\leq -\left(\frac1{\bar{\alpha}}-L\right)\sum_{i=1}^nD_{h_S}\left(x^{k+1}_i,x^k_i\right).
	\end{equation}
	We then estimate $D_{h_S}(x^{k+1}_i,x^k_i)$ for $i\in[n]$, $k\in\cK$. Note that $D_{h_S}^{\prime\prime}(x,x^k_i)=1/x$ and the function $x\mapsto1/x $ is monotonic. For each $x$ lying in the interval between $x^k_i$ and $x^{k+1}_i$, we have 
	\[D_{h_S}^{\prime\prime}(x,x^k_i)\geq \min \left\{\frac1{x^{k+1}_i},\frac1{x^k_i} \right\}.\] Thus, the function $x\mapsto D_{h_S}(x,x^k_i)$ is strongly convex with modulus $\min \{1/{x^{k+1}_i},1/{x^k_i} \}$ in the interval between $x^k_i$ and $x^{k+1}_i$. This, together with $D_{h_S}(x^k_i,x^k_i)=0$ and $D_{h_S}^{\prime}(x,x^k_i)|_{x=x^k_i}=0$, yields 
	\begin{equation*}\label{eq:estimateD1}
		D_{h_S}\left(x^{k+1}_i,x^k_i\right)\geq \frac12\min \left\{\frac1{x^{k+1}_i},\frac1{x^k_i} \right\}\left(x^{k+1}_i-x^k_i\right)^2.
	\end{equation*}
	Note that the bounded iterative ratio property \eqref{eq:beta} implies $ \min \{1/{x^{k+1}_i},1/{x^k_i} \}\geq 1 / (\exp(\bar{\alpha}\beta)x^k_i ) $ for $i\in[n],k\in\cK$. It follows that
	\begin{equation*} 
		D_{h_S}\left(x^{k+1}_i,x^k_i\right)\geq \frac1{2\exp(\bar{\alpha}\beta)}\frac{\left(x^{k+1}_i-x^k_i\right)^2}{x^k_i},\qquad \forall\ i\in[n],~k\in\cK. 
	\end{equation*}
	Combining this estimation with \eqref{eq:estimateD0}, we see that for $k\in\cK$,
	\begin{equation*}
		F(\x^{k+1})-F(\x^k)\leq -\frac{\left(\frac1{\bar{\alpha}}-L\right)}{2\exp\left(\bar{\alpha}\beta\right)}\sum_{i=1}^n\frac{\left(x^{k+1}_i-x^k_i\right)^2}{x^k_i}=-\frac{\kappa_1}{2}\left\|\Diag\left(\sqrt{\x^k}\right)^{-1} \left({\x^{k+1}-\x^k}\right) \right\|^2_2. 
	\end{equation*}
	
	(iii) To begin, we estimate the element-wise relative error $|x^{k+1}_i-x^k_i|$. The update \eqref{eq:explicitbreg} and the fact that $x^k_i>0$ imply 
	\begin{equation*}
		\begin{aligned}
			\left|x^{k+1}_i-x^k_i\right|&=\left|\exp\left(-\alpha_k\left(\nabla f(\x^k)+\A^{\top}\bmu^k\right)_i\right)-1\right|\cdot x^k_i\\
			&\geq \left|\exp\left(-\underline{\alpha}\left(\nabla f(\x^k)+\A^{\top}\bmu^k\right)_i\right)-1\right| \cdot x^k_i,\qquad~ \forall\  i\in[n],
		\end{aligned}
	\end{equation*}
	where the inequality is due to $\alpha_k\geq\underline{\alpha}$ and the fact that the function $t\mapsto|\exp(t\cdot c)-1|$ monotonically increases on $\R_+$ for every $c\in\R$.
	
	Dividing $\sqrt{x_i^k}$ on both sides, we obtain 
	\begin{equation}\label{eq:lbeta2}
		\left|\frac{x^{k+1}_i-x^k_i}{\sqrt{x_i^k}}\right|\geq \left|\exp\left(-\underline{\alpha}\left(\nabla f(\x^k)+\A^{\top}\bmu^k\right)_i\right)-1\right| \cdot \sqrt{x^k_i},\qquad\quad \forall\  i\in[n].
	\end{equation}
	We then estimate the right-hand side of \eqref{eq:lbeta2}. Recall that $\|\nabla f(\x^k)+\A^{\top}\bmu^k\|_2\leq\beta$ for $k\in\cK$ and observe the following simple estimate for $x\in[-\beta,\beta]$:
	\begin{equation*}\label{eq:integral}
		\left|\exp\left(-\underline{\alpha}x\right)-1\right|=\underline{\alpha}\left|\int^{x}_0\exp\left(-\underline{\alpha}s\right){\rm d}s\right|\geq \underline{\alpha}\left|\int^{x}_0\exp\left(-\underline{\alpha}\beta\right){\rm d}s\right|=\underline{\alpha}\exp\left(-\underline{\alpha}\beta\right)|x|=\kappa_2|x|.
	\end{equation*}
	We obtain
	\begin{equation*}\label{eq:lbeta}
		\left|\exp\left(-\underline{\alpha}\left(\nabla f(\x^k)+\A^{\top}\bmu^k\right)_i\right)-1\right|\geq \kappa_2\left|\left(\nabla f(\x^k)+\A^{\top}\bmu^k\right)_i\right|,\qquad \forall~i\in[n],~k\in\cK. 
	\end{equation*}
	Combining this with \eqref{eq:lbeta2}, we see that
		\begin{equation*}
		\left|\frac{x^{k+1}_i-x^k_i}{\sqrt{x_i^k}}\right|\geq \kappa_2\left|\left(\nabla f(\x^k)+\A^{\top}\bmu^k\right)_i\right| \cdot \sqrt{x^k_i},\qquad\quad \forall\  i\in[n],~k\in\cK.
	\end{equation*}
It follows that
	\begin{equation*}
		\begin{aligned}
			\left\|\Diag\left(\sqrt{\x^k}\right)^{-1} \left({\x^{k+1}-\x^k}\right) \right\|_2\geq&  \kappa_2 \left\|\Diag\left(\sqrt{\x^k}\right)\left(\nabla f(\x^k)+\A^{\top}\bmu^k\right)\right\|_2\\
			\geq& \kappa_2 \dist\left(\bz, \Diag\left(\sqrt{\x^k}\right)\partial F\left(\x^k\right)\right),
		\end{aligned}
	\end{equation*}
	where the second inequality uses the fact $\nabla f(\x^k)+\A^{\top}\bmu^k\in\partial F(\x^k)$ given by \eqref{eq:equi-con} and $\x^k\in\R^n_{++}$.
	The proof is complete.
\end{proof}
\subsubsection{Abstract convergence result}
With the above scaled conditions in hand, we are able to develop an abstract convergence result similar to \cite[Lemma 2.6]{attouch2013convergence}, which lies at the core of our analysis. 
\begin{proposition}\label{pro:localconvergence}
	Consider the setting of Proposition \ref{pro:relative_error}. 
	Suppose that the function $F$ satisfies the SK\L\ property at $\x^*$ with objects $\cU$, $\eta$, and $\varphi$. 
	Let $d\in(0,\rho)$ satisfy  $\B(\x^*,d)\subseteq \cU$. Let $r\coloneqq\|\x^*\|_2+d$. Suppose further that 
	\begin{gather}
		F(\x^*)\leq F(\x^0)<F(\x^*)+\eta;\label{eq:localcon1}\\
		\left\|\x^0-\x^*\right\|_2+\sqrt{\frac{2r\left(F(\x^0)-F(\x^*)\right)}{\kappa_1}}+\frac{2\sqrt{r}}{\kappa_1\kappa_2}\varphi\left(F({\x}^{0})-F(\x^*)\right)<d;\label{eq:localcon2}   \\
	 F(\x^{k})\geq F(\x^*),\qquad\forall~k\geq0.\label{eq:localcon3}
	\end{gather}
	Then, the sequence $\{\x^k\}_{k\geq0}$ satisfies
	\begin{gather*}
		\x^k\in \B(\x^*,d), \quad\forall\ k\geq0;\\
		\sum_{k=0}^{\infty}\left\|\x^{k+1}-\x^k\right\|_2<+\infty;\\
		F(\x^k)\rightarrow F(\x^*)
	\end{gather*}
	and converges to a critical point $\bar\x\in \B(\x^*,d)$ with $F(\bar\x)= F(\x^*)$.
\end{proposition} 
\begin{proof}
	The core of this proof is to establish the following inequality for all $K\geq0$ by \textit{induction}:
	\begin{equation}\label{eq:induction}
		\left\|\x^0-\x^*\right\|_2+\sum_{t=0}^{K}\left\|\x^t-\x^{t+1}\right\|_2< d.
	\end{equation}
	We first prove \eqref{eq:induction} for $K=0$. 
	Since $\x^0\in \B(\x^*,d)$ and $d<\rho$, we have $\|\x^0-\x^*\|_2<\rho$. Then, by Proposition \ref{pro:relative_error}, the scaled sufficient decrease \eqref{eq:relative_error1} holds for $k=0$; i.e.,
	\begin{equation*}
		F(\x^{1})-F(\x^0)\leq { -\frac{\kappa_1}{2} }\left\|\Diag\left(\sqrt{\x^0}\right)^{-1}\left({\x^{1}-\x^0}\right)\right\|^2_2. 
	\end{equation*}
	Note that $\|\x^0\|_2\leq\|\x^*\|_2+d=r$ by $\x^0\in\B(\x^*,d)$. We have $x^0_i\leq r$, $i\in[n]$. The above inequality further yields
	\[F(\x^{1})-F(\x^0)\leq { -\frac{\kappa_1}{2r} }\left\|\x^{1}-\x^0\right\|_2^2.\]
	This leads to an upper bound on the distance $\|\x^0-\x^1\|_2$:
	\[\left\|\x^0-\x^1\right\|_2\leq\sqrt{\frac{2r\left(F(\x^0)-F(\x^1)\right)}{\kappa_1}}\leq \sqrt{\frac{2r\left(F(\x^0)-F(\x^*)\right)}{\kappa_1}}, \]
	where the second inequality is due to \eqref{eq:localcon3}. 
	
	The above upper bound, together with \eqref{eq:localcon2}, yields $\|\x^0-\x^*\|_2+\|\x^0-\x^1\|_2<d$. We conclude that \eqref{eq:induction} holds for $K=0$.
	
	Now, suppose that \eqref{eq:induction} holds for $K=k$ for some $k\geq0$, which yields	\[\x^t\in \B(\x^*,d),\qquad t=0,1,\dots,k+1.\]
	 To complete the induction, we need to prove that \eqref{eq:induction} holds for $K=k+1$.  
	
	To begin, we consider a trivial case where $F(\x^{t_0})=F(\x^*)$ for some $t_0\leq k+1$. Observe that the decrease property \eqref{eq:BPGe_suffDescent} ensures $F(\x^{t_0+1})\leq F(\x^{t_0})= F(\x^*)$; and the condition \eqref{eq:localcon3} ensures $F(\x^{t_0+1})\geq F(\x^*)$. We have
	\[F(\x^{t_0+1})=F(\x^*)=F(\x^{t_0}).\] 
	This, together with the decrease property \eqref{eq:BPGe_suffDescent}, yields $\x^{t_0+1}=\x^{t_0}$. Using the first-order optimality condition of the update \eqref{eq:iterbreg} for $k=t_0$, we obtain
	$\bz\in\nabla f(\x^{t_0})+\partial\delta_{\cD_P}(\x^{t_0})$;
	i.e., $\x^{t_0}$ is a critical point of $F=f+\delta_{\cD_P}$. This further implies that
	\[\x^{k}\equiv\x^{t_0}\in\B(\x^*,d),\qquad \forall~k\geq t_0.\]
	This immediately yields the desired results:
	\[\x^k\in\B(\x^*,d),\quad\forall\ k\geq0;\quad \x^k\to\x^{t_0}\in \B(\x^*,d);\quad F(\x^k)\to F(\x^{t_0})=F(\x^*);\]
	\[	\left\|\x^0-\x^*\right\|_2+\sum_{t=0}^{\infty}\left\|\x^t-\x^{t+1}\right\|_2=	\left\|\x^0-\x^*\right\|_2+\sum_{t=0}^{t_0-1}\left\|\x^t-\x^{t+1}\right\|_2< d. \]
	Here, the last inequality is implied by \eqref{eq:induction} with $K=k$.
	
	Therefore, we only need to focus on the non-trivial case, where we have
	\[F(\x^{t})>F(\x^*),\quad t=0,1,\ldots, k+1;\qquad F(\x^{t})\geq F(\x^*),\quad\forall~ t\geq k+2.\] 
	Note that the decrease property of $\{F(\x^k)\}_{k\geq0}$ and condition \eqref{eq:localcon1} ensure $F(\x^{k})\leq F(\x^0)<F(\x^*)+\eta$ for all $k\geq0$. We have the following bounds on $F(\x^t)$:
	\begin{equation}\label{eq:Fcon}
	F(\x^*)<F(\x^{t}) <F(\x^*)+\eta,\quad t=0,1,\ldots, k+1;\quad  F(\x^*)\leq F(\x^{k+2})< F(\x^*)+\eta.
	\end{equation}
The above bounds, together with the inclusion $\x^t\in\B(\x^*,d)\subseteq \cU$ for $t=0,1,\ldots,k+1$, ensure that the SK\L\ inequality of $F$ at $\x^*$ applies to $\x^t$, $t=0,1,\ldots,k+1$. 
Using the concavity of $\varphi$ and the SK\L\ inequality $\varphi^{\prime}(F({\x}^{t})-F(\x^*))\dist(\bz,\Diag(\sqrt{\x^t})\partial F\left(\x^t\right))\geq1$, we see that for $t=0,1,\ldots,k+1$,
\begin{align}\label{eq:KL}			
	\varphi\left(F({\x}^{t})-F(\x^*)\right)-\varphi\left(F({\x}^{t+1})-F(\x^*)\right)&\geq \varphi^{\prime}\left(F({\x}^{t})-F(\x^*)\right)\left(F({\x}^{t})-F(\x^{t+1})\right)\nonumber\\
	&\geq \frac{F({\x}^{t})-F(\x^{t+1})}{\dist\left(\bz,\Diag\left(\sqrt{\x^t}\right)\partial F\left(\x^t\right)\right)}.
\end{align}
Moreover, for $t=0,1,\ldots,k+1$, the inclusion $\x^t\in\B(\x^*,d)$ implies $\|\x^t-\x^*\|_2\leq d<\rho$. Hence, we have $\cK\supseteq \{0,1,\ldots,k+1\}$ for Proposition \ref{pro:relative_error}, yielding the scaled sufficient decrease and scaled relative error for $t=0,1,\ldots,k+1$:
	\begin{equation}\label{eq:estimateD}
		F(\x^{t+1})-F(\x^t)\leq { -\frac{\kappa_1}{2} }\left\|\Diag\left(\sqrt{\x^t}\right)^{-1}\left({\x^{t+1}-\x^t}\right)\right\|^2_2,
	\end{equation}
	\begin{equation}\label{eq:estimate_dist}
		\dist\left(\bz,\Diag\left(\sqrt{\x^t}\right)\partial F(\x^t)\right)\leq \frac1{\kappa_2}\left\|\Diag\left(\sqrt{\x^t}\right)^{-1}\left(\x^{t+1}-\x^t\right)\right\|_2.
	\end{equation}
	Combining  \eqref{eq:KL}---\eqref{eq:estimate_dist}, we see that for $t=0,1,\ldots,k+1$,
	\[\left\|\Diag\left(\sqrt{\x^t}\right)^{-1}\left(\x^{t+1}-\x^t\right)\right\|_2\leq \frac{2}{\kappa_1\kappa_2}\left(\varphi\left(F({\x}^{t})-F(\x^*)\right)-\varphi\left(F({\x}^{t+1})-F(\x^*)\right)\right).\]
	Note that  $\|\x^t\|_2\leq \|\x^*\|_2+d=r$ for $t=0,1\ldots k+1$ by the inclusion $\x^t\in \B(\x^*,d)$.  We have ${x^t_i}\leq {r}$ for $i\in[n]$,  $t=0,1,\dots, k+1$. Hence, the above inequality further implies that
	\[ \left\|\x^{t+1}-\x^t\right\|_2\leq \frac{2\sqrt{r}}{\kappa_1\kappa_2}\left(\varphi\left(F({\x}^{t})-F(\x^*)\right)-\varphi\left(F({\x}^{t+1})-F(\x^*)\right)\right),\quad~t=0,1,\dots, k+1.\]
	Summing up this inequality from $t=0$ to $t=k+1$, we obtain
	\[\begin{aligned}
		\sum_{t=0}^{k+1}\left\|\x^{t+1}-\x^t\right\|_2&\leq \frac{2\sqrt{r}}{\kappa_1\kappa_2}\left(\varphi\left(F({\x}^{0})-F(\x^*)\right)-\varphi\left(F({\x}^{k+2})-F(\x^*)\right)\right)\\
		&\leq \frac{2\sqrt{r}}{\kappa_1\kappa_2}\varphi\left(F({\x}^{0})-F(\x^*)\right),
	\end{aligned} \]
where the last inequality uses  $\varphi\geq0$ on $[0,\eta)$ and $F({\x}^{k+2})-F(\x^*)\in[0,\eta)$ given by \eqref{eq:Fcon}.
 
The above inequality, together with the condition \eqref{eq:localcon2}, implies that
	\begin{equation*}\label{eq:induction2}
		\left\|\x^0-\x^*\right\|_2+\sum_{t=0}^{k+1}\left\|\x^{t+1}-\x^t\right\|_2\leq \left\|\x^0-\x^*\right\|_2+\frac{2\sqrt{r}}{\kappa_1\kappa_2}\varphi\left(F({\x}^{0})-F(\x^*)\right)<d,
	\end{equation*}
	which proves \eqref{eq:induction} for $K=k+1$, completing the \textit{induction}. We conclude that \eqref{eq:induction} holds for all $K\geq0$. It follows that
	\begin{equation*}
		\x^k\in \B(\x^*,d), \quad \forall\ k\geq0;\qquad~\left\|\x^0-\x^*\right\|_2+\sum_{k=0}^{\infty}\left\|\x^{k+1}-\x^k\right\|_2\leq  d < +\infty.
	\end{equation*}
	The finite length of $\{\x^k\}_{k\geq0}$ yields $\x^k\to\bar\x$ for some $\bar\x\in\R^n$ by Cauchy's convergence criterion. By Proposition \ref{pro:escape}, $\bar\x$ is a critical point. Moreover, we have $\bar\x\in \B(\x^*,d)$ by $\x^k\in \B(\x^*,d)$. 
	
	The remaining task is to show $F(\bar\x)=F(\x^*)$. We have proved this result when $F(\x^{t_0})=F(\x^*)$ for some $t_0\geq0$. Hence, it suffices to consider the case where $F(\x^k)>F(\x^*)$ for all $k\geq0$. Recall that $F(\x^k)<F(\x^*)+\eta$ and $\x^k\in\B(\x^*,d)\subseteq\cU$, $k\geq0$. We know that the SK\L{} inequality at $\x^*$ applies to $\x^k$, $k\geq0$; i.e., 
		\begin{equation}\label{eq:final2}
		\varphi^{\prime}\left(F(\x^k)-F(\x^*)\right)\dist\left(\bz,\Diag\left(\sqrt{\x^k}\right)\partial F\left(\x^k\right)\right)\geq1,\qquad\forall~k\geq0.
	\end{equation}
On the other hand, the fact that $\x^k\to\bar\x$ implies 
 $F(\x^k)\to F(\bar\x)$. It follows that $F(\x^{k+1})-F(\x^k)\to0$, which, together with the scaled sufficient descent \eqref{eq:estimateD} and scaled relative error \eqref{eq:estimate_dist} for $t=k$, implies that
	\begin{equation*}\label{eq:final1}
		\dist\left(\bz,\Diag\left(\sqrt{\x^k}\right)\partial F\left(\x^k\right)\right)\to0. 
	\end{equation*}
	Combined with \eqref{eq:final2}, the above implies that
        \[\varphi^{\prime}\left(F(\x^k)-F(\x^*)\right)\to\infty.\]
	Recall that $\varphi$ is continuously differentiable on $(0,\eta)$ and $F(\x^k)-F(\x^*)\leq F(\x^0)-F(\x^*)<\eta$. The above divergence implies that $F(\x^k)-F(\x^*)\to0$. Since $F(\x^k)\to F(\bar\x)$ by $\x^k\to\bar\x$, we conclude that $F(\bar\x)=F(\x^*)$. The proof is complete.
\end{proof}
\begin{corollary}\label{co:localconvergence2}
	Consider the setting of Proposition \ref{pro:relative_error}. Suppose that $\x^*$ is an accumulation point of the sequence $\{\x^k\}_{k\geq0}$ and $F$ has the SK\L\ property at $\x^*$ with objects $\cU,\eta$, and $\varphi$. Then,  there exists an index $K\geq0$ such that 
	\[\x^k\in \cU, \qquad \left\|\x^k-\x^*\right\|_2\leq\rho,\qquad 0\leq F(\x^k)-F(\x^*)<\eta,\qquad \forall~k\geq K. \]
	Furthermore, the sequence $\{\x^k\}_{k\geq0}$ converges to $\x^*$ and $\x^*$ is a critical point of $F$. 
\end{corollary}
\begin{proof}
	Our strategy is to verify that conditions \eqref{eq:localcon1}---\eqref{eq:localcon3} hold when $\x^k$ is replaced by $\x^{K+k}$ for some index $K>0$.
	
	\textbf{Verifying \eqref{eq:localcon3}}:	Since $\x^*$ is an accumulation point of $\{\x^k\}_{k\geq0}$, there is a subsequence $\{\x^{k_t}\}_{t\geq0}$ converging to $\x^*$. It follows that $F(\x^{k_t})\rightarrow F(\x^*)$. This, together with Lemma \ref{le:suffDescent}, which ensures the monotonic decrease of $\{ F(\x^k)\}_{k\geq0}$, implies that
	\[F(\x^{k})\rightarrow F(\x^*);\qquad\qquad F(\x^k)\geq F(\x^*),\quad \forall~ k\geq0.\]
	It follows that $F(\x^{K+k})\geq F(\x^*)$ for all $k\geq0$ and $K\geq0$. Hence, \eqref{eq:localcon3} holds if $\x^k$ is replaced by $\x^{K+k}$ for $K\geq0$.
	
	\textbf{Verifying \eqref{eq:localcon1}}:
	Since $F(\x^{k})\rightarrow F(\x^*)$, there is an index $K_1$ such that \[F(\x^{K})< F(\x^*)+\eta, \qquad \forall~ K\geq K_1.\] Thus, \eqref{eq:localcon1} holds if $\x^0$ is replaced by $\x^K$ for $K\geq K_1$. 
	
	\textbf{Verifying \eqref{eq:localcon2}}: Since $\x^{k_t}\to\x^*$, there is an index $T>0$ such that
	\[	\left\|\x^{k_t}-\x^*\right\|_2+\sqrt{\frac{2r\left(F(\x^{k_t})-F(\x^*)\right)}{\kappa_1}}+\frac{2\sqrt{r}}{\kappa_1\kappa_2}\varphi\left(F({\x}^{k_t})-F(\x^*)\right)<d,\quad \forall\ t\geq T. \]
	Let $K=\min\{k_t:t\geq T,k_t\geq K_1\}$. The above inequality ensures that condition \eqref{eq:localcon2} holds when $\x^0$ is replaced with $\x^K$. We conclude that  conditions \eqref{eq:localcon1}---\eqref{eq:localcon3} hold when $\x^k$ is replaced by $\x^{K+k}$.  
	
	Now, define the new sequence $\{\y^k\}_{k\geq0}$ by $\y^k=\x^{K+k}$. Clearly, $\{\y^k\}_{k\geq0}$ can be considered as a BPGM sequence initialized at $\x^K$.
	Due to the verification above, we can apply Proposition~\ref{pro:localconvergence} to $\{\y^k\}_{k\geq0}$. It follows that 
	$\{\x^k\}_{k\geq0}$ converges to a critical point of $F$. As $\x^*$ is the accumulation point of $\{\x^k\}_{k\geq0}$, we see that $\x^k\to\x^*$ and $\x^*$ is a critical point of $F$. 
	
	Moreover, Proposition \ref{pro:localconvergence} ensures $\y^k=\x^{K+k}\in \B(\x^*,d)$ for $k\geq0$. Since $\B(\x^*,d)\subseteq\cU$ and $d<\rho$, we obtain $\x^k\in \cU$ and $\|\x^k-\x^*\|_2\leq\rho$ for $k\geq K$. Recall that $0\leq F(\x^k)-F(\x^*)<\eta$ for $k\geq K$ from the verification of  \eqref{eq:localcon3} and \eqref{eq:localcon1}. We complete the proof.
\end{proof}
Theorem \ref{th:convergence} directly follows from Corollary \ref{co:localconvergence2}.

\subsection{Proof of Proposition \ref{pro:skl}}
The main idea of the proof is constructing an appropriate composite function and utilizing its \L{}ojasiewicz inequality (see \cite[Theorem 3.1]{bolte2007lojasiewicz}) to derive the SK\L\ property of $H$. 

To begin, we define the element-wise square function $G: \R^n\to\R^n_+$ and the composite function $g: \R^n\to\overline{\R}$ by 
\[G\left(\y\right)\coloneqq\left(y_1^2,y_2^2,\ldots,y^2_n\right), \qquad g(\y)\coloneqq H\left(G\left(\y\right)\right). \]
By \cite[Definition 6.6.1]{facchinei2003finite}, we know that $G$ is a subanalytic function since its graph $\{(\y,\x):\x=\y^2\}$ is a semialgebraic (and thus subanalytic) set. 
Note that $H$ is also subanalytic and $G$ is continuous. By property (p5) on \cite[p. 597]{facchinei2003finite}, the composite function $g:\x\mapsto H(G(\x))$ is also subanalytic. Moreover, the function $g$ is continuous on $\dom(g)$ since $H$ and $G$ are both continuous.

Then, given an arbitrary $\x^*\in\dom(\partial H)\subseteq\R^n_+$, let us find a K\L\ exponent for $g$ at $\y^*\coloneqq\sqrt{\x^*}$. As $g$ is subanalytic and  continuous on $\dom(g)$, by \citep[Theorem 3.1 and Remark 3.2]{bolte2007lojasiewicz}, there exists an exponent $\alpha\in[0,1)$, a constant $c>0$, and a neighborhood $\cY\subseteq\R^n$ of $\y^*$ such that  
\[\left|g(\y)-g(\y^*)\right|^{\alpha}\leq c\cdot\dist\left(\bz,\partial g(\y)\right),\qquad\quad\forall~\y \in \cY,\]
where we use the convention $\dist(\bz,\emptyset)=+\infty$ and $\infty\leq c^{\prime}\cdot\infty$ for all $c^{\prime}>0$.

The above inequality can be written as
\begin{equation}\label{eq:Hsquare}
	\left|H\left(G(\y)\right)-H\left(G(\y^*)\right)\right|^{\alpha}\leq c\cdot\dist\left(\bz,\partial g(\y)\right),\qquad\quad\forall~\y \in \cY.
\end{equation}
Observe that there is a neighborhood $\cX$ of $\x^*$ such that $\sqrt{\x}\in \cY$ for all $\x\in \cX$ by the continuity of the square root function. Let $\y=\sqrt{\x}$ in \eqref{eq:Hsquare} for $\x\in\cX$ and note that $G(\y)=\x$, $G(\y^*)=\x^*$. We obtain 
\begin{equation}\label{eq:gkl}
	\left|H(\x)-H(\x^*)\right|^{\alpha}\leq c\cdot\dist\left(\bz,\partial g\left(\sqrt{\x}\right)\right),\qquad\quad \forall~ \x\in \cX.
\end{equation}
Next, we characterize the subdifferential $\partial g(\sqrt{\x})$.
By \citep[Theorem 10.6]{rockafellar2009variational}, one has
\begin{equation*}\label{eq:hatpartialg}
	\nabla G(\y)^\top\widehat{\partial} H\left(G(\y)\right)\subseteq\widehat{\partial} g(\y) ,\qquad\quad\forall~ \y\in\R^n.    
\end{equation*}
Upon taking an outer limit on both sides and using the fact that $\partial g$ is an outer limit of $\widehat{\partial} g$ due to \citep[Equation 8(5)]{rockafellar2009variational}, we obtain
\begin{equation}\label{eq:partial1}
	\limsup_{\y^{\prime} \underset{g}{\longrightarrow}\y}\nabla G\left(\y^{\prime}\right)^\top\widehat{\partial}H\left(G\left(\y^{\prime}\right)\right)\subseteq \limsup_{\y^{\prime} \underset{g}{\longrightarrow}\y}\widehat{\partial}g\left(\y^{\prime}\right)=\partial g(\y).
\end{equation}
On the other hand, the definition of the outer limit implies that
\begin{equation}\label{eq:partial2}
	\limsup_{\y^{\prime} \underset{g}{\longrightarrow}\y}\nabla G\left(\y^{\prime}\right)^\top\widehat{\partial}H\left(G\left(\y^{\prime}\right)\right)\supseteq \nabla G(\y)^\top\limsup_{\y^{\prime} \underset{g}{\longrightarrow}\y}\widehat{\partial}H\left(G\left(\y^{\prime}\right)\right). 
\end{equation}
Observe that by letting $\x^k=G(\y^k)$ and/or $\y^k=\sqrt{\x^k}$, we have a one-to-one correspondence between the sequences 
\[\left\{G\left(\y^k\right)\right\}_{k\geq0}\ \text{with }\ \y^k\underset{g}{\longrightarrow}\y\qquad \text{and}\qquad \left\{\x^k\right\}_{k\geq0}\subseteq\R^n_+\  \text{with }\ \x^k\underset{H}{\longrightarrow}G(\y).\]  
Hence, the limit on the right-hand side of \eqref{eq:partial2} can be simplified as
\begin{equation}\label{eq:partial3}
	\limsup_{\y^{\prime} \underset{g}{\longrightarrow}\y}\widehat{\partial}H\left(G\left(\y^{\prime}\right)\right)=\limsup_{{\x^{\prime}\underset{H}{\longrightarrow}G(\y) } }\widehat{\partial}H\left(\x^{\prime}\right)={\partial} H\left(G(\y)\right),  
\end{equation}
where the second equality is due to \citep[Equation 8(5)]{rockafellar2009variational}.

Combining \eqref{eq:partial1}---\eqref{eq:partial3}, we obtain
\[\nabla G(\y)^\top{\partial} H\left(G(\y)\right)\subseteq\partial g(\y). \]
Let $\y=\sqrt{\x}$ in above inclusion and note that $\nabla G(\y)=\Diag(2\y)$. Hence, we have
$\Diag(2\sqrt{\x})\partial H(\x)\subseteq \partial g(\sqrt{\x})$.
It follows that 
\[\dist\left(\bz,\partial g(\sqrt{\x})\right)\leq \dist(\bz,\Diag\left(2\sqrt{\x})\partial H(\x)\right).\] This, together with \eqref{eq:gkl}, implies that 
\[\left|H(\x)-H(\x^*)\right|^{\alpha}\leq 2c\cdot\dist\left(\bz,\Diag\left(\sqrt{\x}\right)\partial H(\x)\right),\qquad \forall~\x\in\cX.\]
The above inequality says that the function $H$ satisfies the SK\L\ property at $\x^*$ (with exponent $\alpha$). Since $\x^*$ is an arbitrary point of $\dom(\partial H)$, we conclude that $H$ is an SK\L{} function. 
\section{Linear Convergence of BPGM}\label{sec:linear}
Let us now turn to study the convergence rate of the BPGM with kernel $h_S$ when applied to Problem \eqref{eq:obj}. It is well known that if $g$ is a K\L\ function with exponent $1/2$, then the local convergence rate of a host of iterative methods for minimizing $g$ is at least linear; see, e.g., \cite[Theorem 2]{attouch2009convergence}. By combining the scaled sufficient decrease and relative error conditions in Proposition \ref{pro:relative_error} with
the SK\L\ property in Definition \ref{def:skl}, we can prove a similar result for the BPGM iterates generated according to \eqref{eq:explicitbreg}:
\begin{theorem}
	\label{th:linear}
	Suppose that (A1)---(A3) hold.
	Suppose further that $F=f+\delta_{\cD_P}$  satisfies the SK\L\ property with exponent $1/2$ and the BPGM sequence $\{\x^k\}_{k\geq0}$ is bounded. Then, $\{\x^k\}_{k\geq0}$ converges R-linearly to a critical point of $F$ denoted by $\x^*$, and $\{F(\x^k)\}_{k\geq0}$ converges Q-linearly to $F(\x^*)$.\footnote{We say that a vector sequence $\{\w^k\}_{k\geq0}$ in $\R^n$ converges Q-linearly (resp. R-linearly) to a vector $\w^*\in\R^n$ if there exists a constant $\gamma\in(0,1)$ and an index $K\geq0$ such that $\|\w^{k+1}-\w^*\|_2\leq\gamma\|\w^k-\w^*\|_2$ for all $k\geq K$ (resp. if there exist constants $\gamma\in(0,1)$ and $\Theta>0$ such that $\|\w^k-\w^*\|_2\leq \Theta\cdot \gamma^k$ for all $k\geq0$); see, e.g., \cite[Appendix A.2]{nocedal2006numerical}. }
\end{theorem}
Theorem \ref{th:linear} begs the question of whether functions with an SK\L\ exponent of $1/2$ are common in applications. To address this, one natural direction is to investigate whether such functions arise from those that have a K\L\ exponent of $1/2$. Curiously, the answer is negative.
\begin{example}[Discordance of K\L\ and SK\L\ exponents]\label{example:wkl}
	Consider the quadratic problem
	\begin{equation}\label{eq:example_QP}
		\begin{array}{cl}
			\min & f(x_1,x_2)=\frac12x_1^2 \\
			\text { subject to } & (x_1,x_2)\in \mathcal{D}_P=\left\{\boldsymbol{x}=(x_1,x_2) \in \mathbb{R}^2: x_1+x_2=1,x_1,x_2\geq0 \right\}. 
		\end{array}
	\end{equation}
	By \cite[Theorem 2.1]{luo1993error} and \cite[Theorem 4.1]{li2018calculus}, the K\L\ exponent of $F = f +\delta_{\cD_P}$ is $1/2$. Moreover, it is clear that $\x^*=(0,1)$ is the optimal solution to Problem \eqref{eq:example_QP}.
	However, for any $\x\in\cD_P\cap\R^2_{++}$ in a neighborhood of $\x^*$, we have
	$\Diag(\sqrt{\x})\partial F(\x)=\{( 
	x_1^{3/2},
	0)+\mu(
	\sqrt{x_1},
	\sqrt{x_2}):\mu\in\R\},$
	which implies that $\dist\left(\bz,\Diag(\sqrt{\x})\partial F(\x)\right)=O(x_1^{3/2})$. On the other hand, we have $F(\x)-F(\x^*)=\frac12x_1^2=\Omega(\dist\left(\bz,\Diag\left(\sqrt{\x}\right) \partial F(\x)\right)^{\frac43}),$ which shows that  the SK\L\ exponent of $F$ is at least $3/4$.\end{example}
Still, not all is lost. Indeed, from the first-order optimality conditions of Problem \eqref{eq:example_QP} (see \eqref{kkt-condition}), we know that the multipliers for $\x^*=(0,1)$ are given by $\mu^*=0$ and $\blam^*=\bz$.  In particular, we see that $(\x^*,\blam^*)$ is not a strictly complementary pair; i.e., it fails to satisfy $\x^*+\blam^*\in\R^2_{++}$. As it turns out,
if a critical point $\bar{\x}$ of \eqref{eq:obj} satisfies \textit{strict complementarity} (i.e., there exists a multiplier $\bar{\blam}$ in \eqref{kkt-condition} satisfying $\bar{\x}+\bar\blam\in\R^n_{++}$), we can prove the following result:
\begin{theorem}\label{th:pkl}
	Let ${\x^*}\in\cD_P$ be a critical point of $F=f+\delta_{\cD_P}$. 
	Suppose that the K\L\ exponent of $F=f+\delta_{\cD_P}$ at $\x^*$ is $1/2$ and $\nabla f$ is locally Lipschitz continuous at $\x^*$. Suppose further that strict complementarity holds at $\x^*$. Then, the SK\L\ exponent of $F$ at $\x^*$ is $1/2$.
\end{theorem}
Since many instances of Problem \eqref{eq:obj} are known to give rise to functions with a K\L\ exponent of $1/2$ (see, e.g., \cite{luo1993error,zhou2017unified,li2018calculus}), Theorem \ref{th:pkl} allows us to furnish examples of functions with an SK\L\ exponent of $1/2$. This result, together with Theorem \ref{th:linear}, ensures the linear convergence of the BPGM sequence for a host of problems, further demonstrating the power of our convergence analysis framework.
Moreover, the connection established by Theorem~\ref{th:pkl} is between the two different geometries induced by the Euclidean and Bregman distances, which could be of independent interest.
\subsection{Proof of Theorem \ref{th:linear}}\label{sec:5-1}
By Theorem \ref{th:convergence}, the sequence $\{\x^k\}_{k\geq0}$ converges to a critical point of $F$, which we denote by $\x^*$.
Then, using Corollary \ref{co:localconvergence2} and adopting its notation, there exists an index $K>0$ such that 
\[\x^k\in \cU,\qquad \left\|\x^k-\x^*\right\|_2\leq\rho,\qquad 0\leq F(\x^k)-F(\x^*)<\eta,\qquad\forall~ k\geq K.\] 
The above conditions ensure that the SK\L{} inequality of $F$ at $\x^*$, which has an exponent of $1/2$, applies to $\x^k$, $k\geq K$; i.e.,
\begin{equation}\label{eq:exponent12}
	F(\x^{k})-F(\x^*)\leq c\cdot \dist^2\left(\bz,\Diag\left(\sqrt{\x^k}\right)\partial F\left(\x^{k}\right) \right),\quad \forall\ k\geq K
\end{equation}
for some scalar $c>0$.

Moreover, the inequality $\|\x^k-\x^*\|_2<\rho$ for $k\geq K$ implies $\cK\supseteq\{k\in\N:k\geq K\}$ for Proposition \ref{pro:relative_error}. Hence, the scaled sufficient decease \eqref{eq:relative_error1} and scaled relative error \eqref{eq:relative_error2} hold for $k\geq K$.

Now, we are ready to establish the Q-linear convergence of $\{F(\x^k)\}_{k\geq0}$. By combining the SK\L{} inequality \eqref{eq:exponent12} and the scaled relative error \eqref{eq:relative_error2}, and subsequently using the scaled sufficient decrease \eqref{eq:relative_error1}, we obtain
\begin{align*}
   F(\x^{k})-F(\x^*) &\leq \frac{c}{\kappa_2^2}\left\|\Diag\left(\sqrt{\x^k}\right)^{-1} \left({\x^{k+1}-\x^k}\right) \right\|_2^2 \\
   &\leq \frac{2c}{\kappa_1\kappa_2^2}\left(F(\x^k)-F(\x^{k+1})\right),\quad\forall~ k\geq K.
\end{align*}
Dividing $\frac{2c}{\kappa_1\kappa_2^2}$ on both sides and rearranging the inequality give
\begin{equation}\label{eq:Qlinear}
	F(\x^{k+1})-F(\x^*)\leq \left(1-\frac{\kappa_1\kappa_2^2}{2c}\right)\left(F(\x^k)-F(\x^*)\right),\qquad\forall~ k\geq K.
\end{equation}	
Let $\gamma:=1-{\kappa_1\kappa_2^2}/{(2c)}$. Since $\kappa_1,\kappa_2,c>0$, it is immediate that $\gamma<1$. Moreover, combining \eqref{eq:Qlinear} with the fact that $F(\x^k)-F(\x^*)\geq0$ for all $k\geq K$, we see that $\gamma\geq0$. 
We conclude that $\gamma\in[0,1)$ and $\{F(\x^k)\}_{k\geq0}$ converges Q-linearly to $F(\x^*)$. 

It remains to establish the R-linear convergence of $\{\x^k\}_{k\geq0}$. By the triangle inequality, we have 
\[\left\|\x^k-\x^*\right\|_2\leq\left\|\x^{k^{\prime}+1}-\x^*\right\|_2+ \sum_{t=k}^{k^{\prime}}\left\|\x^{t+1}-\x^{t}\right\|_2\qquad\qquad\forall~k^{\prime}\geq k.\]
Taking $k^{\prime}\to+\infty$ and using the fact that $\x^{k^{\prime}+1}\to\x^*$, we get
\begin{equation}\label{eq:triagnlesum}
	\left\|\x^k-\x^*\right\|_2\leq \lim_{k^{\prime}\to\infty}\left\|\x^{k^{\prime}+1}-\x^*\right\|_2+ \lim_{k^{\prime}\to\infty}\sum_{t=k}^{k^{\prime}}\left\|\x^{t+1}-\x^{t}\right\|_2=\sum_{t=k}^{\infty}\left\|\x^{t+1}-\x^{t}\right\|_2. 
\end{equation}
We then estimate $\|\x^{t+1}-\x^{t}\|_2$ on the right-hand side.
Note that $\|\x^t\|_2\leq r=\|\x^*\|_2+d$ by $\x^t\in\B(\x^*,d)$, $t\geq K$. We have $0<x^t_i\leq r$ for $i\in[n]$, $t\geq K$. Combining this bound and the scaled sufficient decrease \eqref{eq:relative_error1}, we see that
\[	F(\x^{t+1})-F(\x^t)\leq -\frac{\kappa_1}{2r}\left\|\x^{t+1}-\x^t\right\|_2^2,\qquad\forall~t\geq K.\]
This, together with $F(\x^t)\geq F(\x^*)$ for all $t\geq K$, implies that
\[\left\|\x^{t+1}-\x^t\right\|_2\leq\sqrt{\frac{2r\left(F(\x^{t})-F(\x^{t+1})\right)}{\kappa_1}}\leq \sqrt{\frac{2r\left(F(\x^{t})-F(\x^{*})\right)}{\kappa_1}},\qquad\forall~t\geq K.\]
On the other hand, \eqref{eq:Qlinear} yields $F(\x^t)-F(\x^*)\leq \gamma^{t-K}(F(\x^K)-F(\x^*))$ for $t\geq K$.
It follows that
\[\left\|\x^{t+1}-\x^t\right\|_2\leq \sqrt{\frac{2r\left(F(\x^{K})-F(\x^{*})\right)}{\kappa_1}}\gamma^{\frac{t-K}{2}},\qquad\forall~t\geq K.\]
Combining the above inequality and \eqref{eq:triagnlesum}, for $k\geq K$, we have
\[\left\|\x^k-\x^*\right\|_2\leq \sum_{t=k}^{\infty} \sqrt{\frac{2r\left(F(\x^{K})-F(\x^{*})\right)}{\kappa_1}}\gamma^{\frac{t-K}{2}}=\sqrt{\frac{2r\left(F(\x^{K})-F(\x^{*})\right)}{\kappa_1}}\frac{\gamma^{\frac{k-K}{2}}}{1-\sqrt{\gamma}}.\]
This establishes the R-linear convergence of the sequence $\{\x^k\}_{k\geq0}$ and completes the proof.

\subsection{Proof of Theorem \ref{th:pkl}}
\subsubsection{Preliminary observations}
Define the interior and boundary index sets of $\x^*$ by
\[\cI\coloneqq\{i:x^*_i>0\}, \qquad \cJ\coloneqq\{j:x^*_j=0\}.\]
Let $a\coloneqq\frac12\min_{i\in\cI}\{x^*_i\}>0$. For $\x\in\B(\x^*,a)$ and $i\in\cI$, we have $a\geq\left\|\x-\x^*\right\|_2\geq x^*_i-x_i\geq2a-x_i$. This yields a positive lower bound on the interior coordinates of $\x$ around $\x^*$:
	\begin{equation}\label{eq:xJ1}
	x_{i}\geq a;\qquad\forall~\x\in\B(\x^*,a),~i\in\cI.
\end{equation}
\textbf{Trivial case:} When $\x^*$ lies in $\R^n_{++}$; i.e., $\cI=[n]$ and $\cJ=\emptyset$, the lower bound \eqref{eq:xJ1} ensures that
\[\dist\left(\bz,\Diag\left(\sqrt{\x}\right) \partial F(\x) \right)\geq \sqrt{a}\cdot\dist\left(\bz,\partial F(\x)\right),\qquad \forall\ \x\in \B(\x,a).\]
This, together with the K\L\ exponent $1/2$ of $F$ at $\x^*$, implies that there exist scalars $c,\eta>0$ and a neighborhood $\cU$ of $\x^*$ such that for all $\x\in\cU$ with $0< F(\x)-F(\x^*)<\eta$,
\[ F(\x)-F(\x^*)\leq c\cdot \dist^2\left(\bz,\partial F(\x)\right)\leq\frac{c}{a}\dist^2\left(\bz,\Diag\left(\sqrt{\x}\right) \partial F(\x) \right).\]
That is, the SK\L\ exponent $1/2$ of $F$ holds at $\x^*$ for the case where $\cI=[n]$ and $\cJ=\emptyset$. 

\textbf{Observations for non-trivial case:} Then, let us focus on the non-trivial case where $\x^*$ lies on the boundary of $\R^n_+$; i.e., $\cJ\neq\emptyset$. We may also assume $\cI\neq\emptyset$ because the proof for $\cI=\emptyset$ is the same except the notation. 

Observe that when $\cJ\neq\emptyset$, the relation \[\dist\left(\bz,\partial F(\x)\right)=\cO\left(\dist\left(\bz,\Diag\left(\sqrt{\x}\right)\partial F(\x)\right)\right)\]
 may fail to hold near $\x^*$ since ${\x_{\cJ}}$ can be arbitrarily close to ${\x^*_{\cJ}}=\bz$. Consequently, the arguments used for the trivial case are no longer applicable. To connect the K\L{} and SK\L{} exponents, we construct an intermediate point $\hat{\x}$ by projecting $\x\in\cD_P$ onto the set $\cS\coloneqq\cD_P\cap\{\x:\x_{\cJ}=\bz\}$; i.e.,
\[\hat{\x}\coloneqq \Pi_{\cS}(\x). \]
We next present several preliminary properties of $\hat{\x}$, which will prove useful for subsequent analysis.
 \begin{fact}\label{fact:hoffman}
 	There exists a scalar $L_H>0$ such that
 	\begin{align}
 			\left\|\hat{\x}-\x\right\|_2&\leq L_H\left\|\x_{\cJ}\right\|_2,~&&\forall~\x\in\cD_P; \label{eq:hatxJ1} \\
 				\left\|\hat{\x}-\x^*\right\|_2&\leq(1+L_H)\left\|\x-\x^*\right\|_2,~&&\forall~\x\in \cD_P; \label{eq:hatxx*} \\
 					\hat{\x}_{\cI}&\geq a\1_{|\cI|},~&&\forall~\x\in\cD_P\cap\B\left(\x^*,\frac{a}{1+L_H}\right). \label{eq:hatxJ} 
 	\end{align}
 \end{fact}
The inequality \eqref{eq:hatxJ1} is nothing but a Hoffman error bound; \eqref{eq:hatxx*} ensures that $\hat{\x}$ remains close to the base point $\x^*$ when $\x$ is near  $\x^*$; and \eqref{eq:hatxJ}  guarantees a positive lower bound on the  interior coordinates of $\hat{\x}$ when $\x$ is in a neighborhood of $\x^*$. We prove them in order.
\begin{proof}[Proof of Fact \ref{fact:hoffman}]
 (i) Observe that  $\cS\neq\emptyset$ due to the fact $\x^*\in\cS$. Moreover, $\cS$ can be characterized by the polyhedron
 \[\cS=\cD_P\cap\left\{\x:\x_{\cJ}=\bz\right\}=\left\{\x\in\R^n:\x_{\cI}\geq0,\x_{\cJ}=\bz,\A\x=\b\right\}.\] 
 By the Hoffman error bound \cite{hoffman2003approximate} (see also \cite[Theorem 16.2]{luo2000error}), there exists a constant $L_H>0$ such that for all $\x\in \R^n$, we have
 \[\left\|\hat{\x}-\x\right\|_2=\dist\left(\x,\cS\right)\leq L_H\left(\left\|\Pi_{\R^{|\cI|}_+}(-\x_{\cI})\right\|_2+\left\|\x_{\cJ}\right\|_2+\left\|\A\x-\b\right\|_2\right).\]
 Consider $\x\in\cD_P$ for the above inequality, where we have $\A\x-\b=\bz$ and $\Pi_{\R^n_+}(-\x_{\cI})=\bz$.  We obtain the desired error bound \eqref{eq:hatxJ1}.

(ii) Combining  the triangle inequality $\|\hat{\x}-\x^*\|_2\leq\|\hat{\x}-\x\|_2+\|\x-\x^*\|_2$ with the error bound \eqref{eq:hatxJ1}, we obtain
 \begin{equation*}
 	\left\|\hat{\x}-\x^*\right\|_2
 	\leq L_H\left\|\x_{\cJ}\right\|_2+\left\|\x-\x^*\right\|_2,\qquad\forall~\x\in \cD_P.
 \end{equation*}
Note that $\|\x_{\cJ }\|_2\leq  \|\x-\x^*\|_2$ by $\x^*_{\cJ}=\bz$.
 The above inequality yields \eqref{eq:hatxx*}.
 
 (iii) For $\x\in\cD_P\cap\B(\x^*,a/(1+L_H))$, by \eqref{eq:hatxx*} we have $a\geq\|\hat{\x}-\x^*\|_2\geq x^*_i-\hat{x}_i$, $i\in[n]$, which, together with $x^*_i\geq2a$ for $i\in\cI$, yields \eqref{eq:hatxJ}.
\end{proof}

 With the preliminary properties above, we now establish the following proposition, which lies at the core of our arguments. As its proof is technical and lengthy, we defer it to the Appendix.  We then proceed to prove Theorem \ref{th:pkl}. 
 \begin{proposition}\label{pro:pkl}
 	Consider the setting of Theorem \ref{th:pkl}. There exist scalars $\tilde{\epsilon}\in(0,a/(1+L_H)]$, $b$, $L_f>0$ such that for $\x\in\cD_P\cap\B(\x^*,\tilde{\epsilon})$ and $\bmu\in\R^m$ satisfying $\|(\nabla f(\x)+\A^{\top}\bmu)_{\cI}\|_2\leq\tilde{\epsilon}$,
 	\begin{equation}\label{eq:conb1}
 		\left\|\Diag\left(\sqrt{\x_{\cJ}}\right)\left(\nabla f(\x)+\A^{\top}\bmu\right)_{\cJ}\right\|_2\geq b\sqrt{\|\x_{\cJ}\|_2},
 	\end{equation} 
 	\begin{equation}\label{eq:disthatx}
 		\dist\left(\bz,\partial F(\hat{\x}) \right)\leq \left\|(\nabla f(\x)+\A^{\top}\bmu)_{\cI}\right\|_2+L_fL_H\left\|\x_{\cJ}\right\|_2.
 	\end{equation}
 \end{proposition}

\subsubsection{Main proof}
Let $c_0,\epsilon_0,\eta_0 > 0$ be the constants from the K\L\ inequality of $F$ at $\x^*$, which satisfy
\begin{equation}\label{eq:KL12}
\begin{aligned}
    &F(\x^{\prime})-F(\x^*)\leq c_0\dist(\bz,\partial F({\x}^{\prime}))^2,\\ \forall~\x^{\prime}\in&\B(\x^*,\epsilon_0)\cap\left\{\x^{\prime}:F(\x^*)<F(\x^{\prime})<F(\x^*)+\eta_0\right\}.
\end{aligned}
\end{equation}
Let $\tilde{\epsilon} > 0$ be the constant in Proposition \ref{pro:pkl}.
Define the local Lipschitz constant of $f$ by 
\[L_0\coloneqq\max\left\{\frac{\left|f(\x)-f(\y)\right|}{\left\|\x-\y\right\|_2}:\x,\y\in \B(\x^*,a)\text{ with }\x\neq\y\right\},\]
which must be finite since $f$ is a continuously differentiable function by (A2). 

Based on the above constants, we aim to find $c,\epsilon,\eta > 0$ such that the SK\L\ inequality 
\[F(\x)-F(\x^*)\leq c\cdot\dist^2\left(\bz,\Diag\left(\sqrt{\x}\right)\partial F(\x)\right)\]
holds whenever 
\[\x\in \cN(\x^*)\coloneqq\left\{\x\in\cD_P:\|\x-\x^*\|_2<\epsilon,~F(\x^*)<F(\x)<F(\x^*)+\eta\right\}.\]  
 Let us define the scalars $\epsilon$ and $\eta$ by
\begin{equation}\label{eq:constant}
	\epsilon=\frac12\cdot\min\left\{\frac{\eta_0}{L_0(1+L_H)},\frac{\epsilon_0}{1+L_H},\tilde{\epsilon}\right\},\qquad\quad \eta=\epsilon.
\end{equation} 
The task is to find a scalar $c>0$ such that for all $\x\in\cN(\x^*)$ and $\w\in\partial F(\x)$,
\begin{equation}\label{eq:skl_desire}
	F(\x)-F(\x^*)\leq c\cdot\left\|\Diag\left(\sqrt{\x}\right)\w\right\|_2^2.
\end{equation}
For this purpose, we divide the subdifferential set $\partial F(\x)$ into two parts:
\[\cW_{\epsilon}^+(\x)\coloneqq\left\{\w\in\partial F(\x):\left\|\w_{\cI}\right\|_2>\epsilon\right\},\quad \cW_{\epsilon}^-(\x)\coloneqq\left\{\w\in\partial F(\x):\left\|\w_{\cI}\right\|_2\leq\epsilon\right\}.\] Clearly, $\partial F(\x)=\cW_{\epsilon}^+(\x)\cup \cW_{\epsilon}^-(\x)$. Hence, it suffices to find scalars $c_1,c_2>0$ such that
\begin{enumerate}
	\item[(E1)] the inequality \eqref{eq:skl_desire} with $c=c_1$ holds for all $\w\in \cW_{\epsilon}^+(\x)$, $\x\in\cN(\x^*)$;
	\item[(E2)] the inequality \eqref{eq:skl_desire} with $c=c_2$ holds for all $\w\in \cW_{\epsilon}^-(\x)$, $\x\in\cN(\x^*)$,
\end{enumerate}
and the proof would be complete by letting $c=\max\{c_1,c_2\}$.

Let us prove (E1) and (E2) in order.

\textbf{Proof of (E1):} For $\x\in \cN(\x^*)\subseteq\B(\x^*,\epsilon)$, the choice \eqref{eq:constant} ensures that $\x\in\B(\x^*,\tilde{\epsilon})$. Using $\tilde{\epsilon}\in(0,a/(1+L_H)]$ and \eqref{eq:xJ1}, we see that $x_{i}\geq a$ for $i\in\cI$, $\x\in\cN(\x^*)$.  It follows that
\[\left\|\Diag\left(\sqrt{\x}\right)\w\right\|_2\geq\left\|\Diag\left(\sqrt{\x_{\cI}}\right)\w_{\cI}\right\|_2\geq \sqrt{a}\left\|\w_{\cI}\right\|_2\geq \sqrt{a}\epsilon,\quad \forall~\x\in \cN(\x^*),~\w\in\cW^+_{\epsilon}(\x).\]
Combining this with the inequality $F(\x)-F(\x^*)<\eta=\epsilon$ for $\x\in\cN(\x^*)$, we obtain
\[F(\x)-F(\x^*)<\epsilon=  \frac{1}{a\epsilon}\cdot a\epsilon^2\leq  \frac{1}{a\epsilon}\left\|\Diag\left(\sqrt{\x}\right)\w\right\|_2^2,\qquad\forall\ \x\in\cN(\x^*),\w\in\cW^+_{\epsilon}(\x). \]
This proves (E1) with $c_1=1/(a\epsilon)$.

\textbf{Proof of (E2):} 
We start from the following decomposition for $\x\in\cD_P$: 
\begin{equation}\label{eq:decomposition}
	\begin{aligned}
		F(\x)-F(\x^*)&=F(\x)-F(\hat{\x})+F(\hat{\x})-F(\x^*)\\
		&=f(\x)-f(\hat{\x})+F(\hat{\x})-F(\x^*).
	\end{aligned}
\end{equation}
Then, we estimate $f(\x)-f(\hat{\x})$ and $F(\hat{\x})-F(\x^*)$ on the right-hand side.  By \eqref{eq:hatxx*}, we have
\begin{equation*}
	\left\|\hat{\x}-\x^*\right\|_2<(1+L_H)\left\|\x-\x^*\right\|_2<(1+L_H)\epsilon,\qquad\forall~\x\in\cN(\x^*).
\end{equation*}
This, together with our choice \eqref{eq:constant} on $\epsilon$ and the bound $\tilde{\epsilon}\leq a/(1+L_H)$, implies that
\begin{equation}\label{eq:hatxx*2}
\left\|\hat{\x}-\x^*\right\|_2<\min\left\{\frac{\eta_0}{L_0},\epsilon_0,a\right\},\qquad\forall~\x\in\cN(\x^*);\qquad \cN(\x^*)\subseteq\B(\x^*,\epsilon)\subseteq\B(\x^*,a).
\end{equation}
 The above inequality and inclusion say that both $\hat{\x}$ and $\x$ are in the ball $\B(\x^*,a)$ when $\x\in\cN(\x^*)$. Recalling the $L_0$-Lipschitz continuity of $f$ in the ball $\B(\x^*,a)$, we have
\begin{equation}\label{eq:f_f}
	f(\x)-f(\hat{\x})\leq L_0\left\|\hat{\x}-\x\right\|_2\leq L_0L_H\left\|\x_{\cJ}\right\|_2,\qquad\forall~\x\in\cN(\x^*),
\end{equation}
where the last inequality is due to the error bound \eqref{eq:hatxJ1}. Moreover, we have
\begin{equation}\label{eq:f_f2}
	F(\hat{\x})-F(\x^*)=f(\hat{\x})-f(\x^*)\leq L_0\left\|\hat{\x}-\x^*\right\|_2<\eta_0,\qquad\forall~\x\in\cN(\x^*), 
\end{equation}
 where the last inequality is due to \eqref{eq:hatxx*2}.

Furthermore,  \eqref{eq:hatxx*2} and \eqref{eq:f_f2} ensure that the K\L\ inequality of $F$  at $\x^*$; i.e., \eqref{eq:KL12}, applies to $\hat{\x}$ when $\x\in\cN(\x^*)$.
We obtain
\begin{equation}\label{eq:FKL12}
	F(\hat{\x})-F(\x^*)\leq c_0\dist^2\left(\bz,\partial F(\hat{\x})\right),\qquad \forall~\x\in\cN(\x^*).
\end{equation}
Combining \eqref{eq:decomposition},  \eqref{eq:f_f}, and \eqref{eq:FKL12} gives
\begin{equation}\label{eq:pwkl1}
		F(\x)-F(\x^*)\leq L_0L_H\left\|\x_{\cJ}\right\|_2+c_0\cdot\dist^2\left(\bz,\partial F(\hat{\x})\right), \qquad\ \forall\ \x\in\cN(\x^*).
\end{equation}
We then seek an upper bound on $\dist(\bz,\partial F(\hat{\x}))$. According to \eqref{eq:equi-con}, every vector $\w\in\partial F(\x)$ for $\x\in\cD_P$ can be written as $\w=\nabla f(\x)+\A^{\top}\bmu-\blam$ for some $\bmu\in\R^m$ and $\blam\in\R^n_+$ with $\blam^{\top}\x=0$. Note that $\x_\cI\geq a$ for $\x\in\cN(\x^*)\subseteq\B(\x^*,a)$ by \eqref{eq:xJ1}. We see that $\blam_{\cI}=\bz$ and 
\[\w_{\cI}=(\nabla f(\x)+\A^{\top}\bmu)_{\cI},\qquad\forall~\w\in \partial F(\x),~\x\in\cN(\x^*).\]
 It follows that 
 \begin{equation*}\label{eq:pklCon}
 \|(\nabla f(\x)+\A^{\top}\bmu)_{\cI}\|_2=\|\w_{\cI}\|_2\leq\epsilon<\tilde{\epsilon},\quad\forall ~\w\in\cW^-_{\epsilon}(\x),~\x\in\cN(\x^*).
 \end{equation*}
 This inequality, together with the inclusion $\cN(\x^*)\subseteq\cD_P\cap\B(\x^*,\tilde{\epsilon})$, ensures that the conditions of Proposition \ref{pro:pkl} are satisfied and \eqref{eq:conb1}, \eqref{eq:disthatx} hold when $\x\in\cN(\x^*)$, $\w\in \cW^-_{\epsilon}(\x)$.

Now, applying the Cauchy--Schwarz inequality to \eqref{eq:disthatx} and using $\w_{\cI}=(\nabla f(\x)+\A^{\top}\bmu)_{\cI}$, we obtain an upper bound on  $\dist(\bz,\partial F(\hat{\x}))$:
\begin{equation*}\label{eq:distFj} 
	\dist\left(\bz,\partial F(\hat{\x})\right)^2\leq 2\|\w_{\cI}\|_2^2+2L_f^2L_H^2\|\x_{\cJ}\|_2^2,\qquad\forall~\x\in\cN(\x^*),~\w\in\cW^-_{\epsilon}(\x).
\end{equation*}
This, together with \eqref{eq:pwkl1}, implies that for $\x\in\cN(\x^*),\w\in\cW^-_{\epsilon}(\x)$,
\begin{equation}\label{eq:pwkl2}
	F(\x)-F(\x^*)\leq L_0L_H\left\|\x_{\cJ}\right\|_2+2c_0\left\|\w_{\cI}\right\|_2^2+2c_0L_f^2L_H^2\left\|\x_{\cJ}\right\|_2^2.
\end{equation}
We then estimate the right-hand side of \eqref{eq:pwkl2}.
First, it follows from \eqref{eq:xJ1} that
\begin{equation}\label{eq:pwkl3}
	\left\|\w_{\cI}\right\|_2\leq\frac{1}{\sqrt{a}} \left\|\Diag(\sqrt{\x_{\cI}})\w_{\cI}\right\|_2\leq\frac{1}{\sqrt{a}} \left\|\Diag(\sqrt{\x})\w\right\|_2,\qquad\forall~\x\in\cN(\x^*)\subseteq\B(\x^*,a). 
\end{equation} 
Second, recall that \eqref{eq:conb1} in Proposition \ref{pro:pkl} holds when $\x\in\cN(\x^*)$, $\w\in\cW^-_{\epsilon}(\x)$. We have
\begin{equation*}\label{eq:pwkl4}
	\left\|\x_{\cJ}\right\|_2\leq\frac1{b^2}\left\|\Diag\left(\sqrt{\x_{\cJ}}\right)\left(\nabla f(\x)+\A^{\top}\bmu\right)_{\cJ}\right\|^2_2,\qquad\ \forall\ \x\in\cN(\x^*),\w\in\cW^-_{\epsilon}(\x). 
\end{equation*}
Recall the expression $\w=\nabla f(\x)+\A^{\top}\bmu-\blam$ for $\w\in\partial F(\x)$, $\x\in\cD_P$, where $x_i\lambda_i=0$, $i\in[n]$. We have $\sqrt{x_i}\cdot\lambda_i=\sqrt{x_i\lambda_i}\cdot\sqrt{\lambda_i}=0$ for $i\in[n]$, and hence $\sqrt{x_i}w_i=\sqrt{x_i}(\nabla f(\x)+\A^{\top}\bmu)_i$, $i\in[n]$.
The above inequality can be formulated as
\begin{equation}\label{eq:pwkl5}
	\left\|\x_{\cJ}\right\|_2\leq\frac1{b^2}\left\|\Diag\left(\sqrt{\x_{\cJ}}\right)\w_{\cJ}\right\|_2^2,\qquad\ \forall\ \x\in\cN(\x^*),\w\in\cW^-_{\epsilon}(\x). 
\end{equation}
Recall that $\x^*_{\cJ}=\bz$ implies
$\|\x_{\cJ}\|_2\leq\|\x-\x^*\|_2\leq\epsilon$ for $\x\in\cN(\x^*)$. Combined with \eqref{eq:pwkl5}, this yields
\begin{equation}\label{eq:pwkl6}
	\|\x_{\cJ}\|^2_2\leq\epsilon\|\x_{\cJ}\|_2\leq \frac{\epsilon}{b^2}\left\|\Diag\left(\sqrt{\x}\right)\w\right\|_2^2,\qquad\ \forall\ \x\in\cN(\x^*),~\w\in\cW^-_{\epsilon}(\x). 
\end{equation}
Combining \eqref{eq:pwkl2}---\eqref{eq:pwkl6}, we see that for $\x\in\cN(\x^*)$ and $\w\in\cW^-_{\epsilon}(\x)$, 
\[\begin{aligned}
	&F(\x)-F(\x^*)\\
	\leq& \frac{L_0L_H}{b^2}\left\|\Diag\left(\sqrt{\x}\right)\w\right\|_2^2 +\frac{2c_0}a\left\|\Diag\left(\sqrt{\x}\right)\w\right\|_2^2 +\frac{2c_0\epsilon L_f^2L_H^2}{b^2} \left\|\Diag\left(\sqrt{\x}\right)\w\right\|_2^2\\
	=&\left(\frac{L_0L_H}{b^2}+\frac{2c_0}a +\frac{2c_0\epsilon L_f^2L_H^2}{b^2}\right)\left\|\Diag\left(\sqrt{\x}\right)\w\right\|_2^2.\\
\end{aligned}\]
This proves (E2) with $c_2=\frac{L_0L_H}{b^2}+\frac{2c_0}a +\frac{2c_0\epsilon L_f^2L_H^2}{b^2}$. The proof is complete.

\section{Conclusion}\label{sec:end}
In this paper,  we took a {  first} step towards resolving the open problem concerning BPGM iterate convergence, showing that the BPGM iterates converge to a critical point for a broad class of problems under the widely adopted Shannon entropy kernel {  and linear constraints}. We developed a novel convergence analysis framework based on a scaled geometry, where the key ingredient is the newly introduced SK\L\ property. We proved that the  SK\L\ property holds for all continuous subanalytic functions and implies the iterate convergence of BPGM. Furthermore, we showed that the BPGM iterates exhibit linear convergence if the SK\L\ exponent is $1/2$. Lastly, we proved that in the setting considered in this paper, functions with K\L\ exponent $1/2$ also have SK\L\ exponent $1/2$ under a strict complementarity condition, thereby furnishing examples of functions with SK\L\ exponent $1/2$. Our work opens up a number of research directions, such as extending the convergence results to cover more general problem settings, studying the convergence behavior of other Bregman distance-based methods, and further elucidating the relationship between the K\L{} and SK\L{}
properties.


\appendix
\section*{Appendix}
\section{Proof of Lemma \ref{le:omit}}
We prove the existence of such $\rho$ and $\beta$ by contradiction. Suppose there is a sequence $\{\y^k\}_{k\geq0}\subseteq \cD_P\cap\R^n_{++}$ (not necessarily generated by the BPGM) converging to $\x^*$ such that
\[\left\|\nabla f(\y^k)+\A^{\top}\bmu^k\right\|_2\to+\infty,\]
where 
$\bmu^k=\argmin_{\bmu\in\R^m}\{\frac1{\alpha_k}\exp(-\alpha_k(\nabla f(\y^k)+\A^{\top}\bmu))^{\top}\y^k+\b^{\top}\bmu\}$ and $\alpha_k>0$. Note that $\{\nabla f(\y^k)\}_{k\geq0}$ is bounded due to $\y^k\to\x^*$ and the continuous differentiability of $f$. The divergence $\|\nabla f(\y^k)+\A^{\top}\bmu^k\|_2\to+\infty$ yields $\|\bmu^k\|_2\to+\infty$. By passing to a subsequence if necessary, we
assume that
\[\frac{\bmu^k}{\|\bmu^k\|_2}\to \bmu^*\in\R^m.\] 
We then seek a contradiction. Let $T_{\alpha}(\y)$ denote the next iterate of the BPGM at $\y\in\cD_P\cap\R^n_{++}$ with kernel $h=h_S$ and step size $\alpha$ in \eqref{eq:obj}. 
By Proposition \ref{pro:explicit}, we have 
\begin{equation}\label{eq:T}
	T_{\alpha_k}(\y^k)=\Diag\left(\exp\left(-\alpha_k \left(\nabla f(\y^k)+\A^{\top}\bmu^k\right) \right)\right)\y^k\in\cD_P\cap\R^n_{++}.
\end{equation}
Observe that $T_{\alpha_k}(\y^k)\in\cD_P$ and $\y^k\in\cD_P$ yield $\A T_{\alpha_k}(\y^k)=\A\y^k=\b$. It follows that
\begin{equation}\label{eq:contra}
	\left(T_{\alpha_k}(\y^k)-\y^k\right)^{\top}\A^{\top}\bmu^*=\left(\b^{\top}-\b^{\top}\right)\bmu^*=0,\qquad \forall\ k\geq0.
\end{equation}
Define the positive and negative index sets of $\A^{\top}\bmu^*$ by
\[\cI^*_+\coloneqq\left\{i:\left(\A^{\top}\bmu^*\right)_i>0\right\}, \qquad \cI^*_-\coloneqq\left\{i:\left(\A^{\top}\bmu^*\right)_i<0\right\}.\]
{ Note that $\bA$ is of full row rank and $\|\bmu^*\|_2=1$ by definition. We have $\A^{\top}\bmu^*\neq\bz$ and hence $\cI^*_+\cup \cI^*_-\neq\emptyset$.
 Using these} index sets and the expression \eqref{eq:T}, we rewrite \eqref{eq:contra} as
\begin{equation}\label{eq:contra2}
	\sum_{i\in{\cI^*_+\cup \cI^*_-}}\left(\exp\left(-\alpha_k\left(\nabla f(\y^k)+\A^{\top}\bmu^k\right)_i\right)y^k_{i}-y^k_{i}\right)\left(\A^{\top}\bmu^*\right)_{i}=0,\qquad \forall\ k\geq0.
\end{equation} 
Our strategy is to show that \eqref{eq:contra2} contradicts the divergence $\|\bmu^k\|_2\to+\infty$.
Note that
\[\lim\limits_{k\to\infty}\frac{(\A^{\top}\bmu^k)_i}{\|\bmu^k\|_2 }=\left(\A^{\top}\frac{\bmu^k}{\|\bmu^k\|_2}\right)_i=\left(\A^{\top}\bmu^*\right)_i;\]
\[ \left(\A^{\top}\bmu^*\right)_i>0,\quad\forall~i\in\cI^*_+;\qquad \left(\A^{\top}\bmu^*\right)_i<0,\quad\forall~i\in\cI^*_-. \] 
These, together with $\|\bmu^k\|_2\to+\infty$, imply that
\[\left(\A^{\top}\bmu^k\right)_i\rightarrow+\infty,\quad \forall\ i\in \cI^*_+;\qquad\quad \left(\A^{\top}\bmu^k\right)_i\rightarrow-\infty,\quad \forall\ i\in \cI^*_-.\] 
Recall the boundedness of $\{\nabla f(\y^k)\}_{k\geq0}$. The above divergence implies that there exists a sufficiently large index $K>0$ such that for all $k\geq K$,
\begin{equation*}\label{eq:amui}
	\left(\nabla f(\y^k)+\A^{\top}\bmu^k\right)_i>0,\quad \forall\ i\in \cI^*_+;\qquad \left(\nabla f(\y^k)+\A^{\top}\bmu^k\right)_i<0,\quad \forall\ i\in \cI^*_-.
\end{equation*}
Combining the above inequality and the facts that $(\A^{\top}\bmu^*)_i>0$ for $i\in \cI^*_+$,   $ (\A^{\top}\bmu^*)_i<0$ for $i\in\cI^*_-$, and $y^k_i>0$ for $i\in[n]$, we see that for both   $i\in \cI^*_+$ and $i\in\cI^*_-$,
\[\left(\exp\left(-\alpha_k\left(\nabla f(\y^k)+\A^{\top}\bmu^k\right)_i\right)y^k_{i}-y^k_{i}\right)\left(\A^{\top}\bmu^*\right)_{i} <0,\qquad\forall~k\geq K.\] 
It follows that for all $k\geq K$, we have
\[ \sum_{i\in{\cI^*_+\cup \cI^*_-}}\left(\exp\left(-\alpha_k\left(\nabla f(\y^k)+\A^{\top}\bmu^k\right)_i\right)y^k_{i}-y^k_{i}\right)\left(\A^{\top}\bmu^*\right)_{i}  <0,\]
which contradicts \eqref{eq:contra2}. This completes the proof.

\section{Proof of Proposition \ref{pro:escape}} \label{appen:escape}
Invoking the update \eqref{eq:explicitbreg} $k$ times, we have
\[\x^k=\Diag\left(\exp\left(-\sum_{t=0}^{k-1}\alpha_t \left(\nabla f(\x^t)+\A^{\top}\bmu^t\right)\right)\right)\x^0,\]
or equivalently 
\[\sum_{t=0}^{k-1}\alpha_t \left(\nabla f(\x^t)+\A^{\top}\bmu^t\right)=\log\left({\x^0}\right)-\log\left({\x^k}\right).\] 
Dividing $\sum_{t=0}^{k-1}\alpha_t$ on both sides, we further have
\begin{equation}\label{eq:tk}
	\frac{\sum_{t=0}^{k-1}\alpha_t\left(\nabla f(\x^t)+\A^{\top}\bmu^t\right)}{\sum_{t=0}^{k-1}\alpha_t} =\frac{1}{\sum_{t=0}^{k-1}\alpha_t}\left(\log\left(\x^0\right)-\log\left(\x^k\right)\right).
\end{equation}
We show that the limit of \eqref{eq:tk} is nothing but the first-order optimality condition \eqref{kkt-condition} at $\bar\x$ to complete the proof.

First, we show that $\{{\sum_{t=0}^{k-1}\alpha_t\nabla f(\x^t)}/{\sum_{t=0}^{k-1}\alpha_t}\}_{k\geq0}$ converges to $\nabla f(\bar\x)$.  By the triangle inequality and $\underline{\alpha}\leq \alpha_k\leq\bar{\alpha}$, we have
\[\begin{aligned}
	\lim\limits_{k\to\infty}\left\|\frac{\sum_{t=0}^{k-1}{\alpha_t}\nabla f(\x^t)}{\sum_{t=0}^{k-1}\alpha_t}- \nabla f(\bar\x)\right\|_2&\leq \lim\limits_{k\to\infty}\frac{\sum_{t=0}^{k-1}\alpha_t \left\|\nabla f(\x^t)-\nabla f(\bar\x)\right\|_2}{\sum_{t=0}^{k-1}\alpha_t}\\
	&\leq \lim\limits_{k\to\infty}\frac{\sum_{t=0}^{k-1} \left\|\nabla f(\x^t)-\nabla f(\bar\x)\right\|_2}{k}\cdot \frac{\bar{\alpha}}{\underline{\alpha}}\\
	&=0,    
\end{aligned}
\]
where the equality is due to \cite[Problem 3-1]{mattuck1999introduction} and $\|\nabla f(\x^k)- \nabla f(\bar\x)\|_2\to0$. 

It follows that
\begin{equation}\label{eq:limitnabla}
	\frac{\sum_{t=0}^{k-1}{\alpha_t}\nabla f(\x^t)}{\sum_{t=0}^{k-1}\alpha_t}\to \nabla f(\bar\x).
\end{equation}
Then, we consider the limit of $\{{\sum_{t=0}^{k-1}\alpha_t\bmu^t}/{\sum_{t=0}^{k-1}\alpha_t}\}_{k\geq0}$. We first show its boundedness. 
By Lemma \ref{le:omit} and the fact that $\x^k\to\bar\x$, we have the boundedness of $\{\nabla f(\x^k)+\A^{\top}\bmu^k\}_{k\geq0}$. This, together with the boundedness of $\{\nabla f(\x^k)\}_{k\geq0}$ and the full column rank of $\A^{\top}$, yields the boundedness of $\{\bmu^k\}_{k\geq0}$. Recalling that $0<\underline{\alpha}\leq\alpha_k\leq\bar{\alpha}$ for all $k\geq0$, we obtain the boundedness of $\{{\sum_{t=0}^{k-1}\alpha_t \bmu^t}/{\sum_{t=0}^{k-1}\alpha_t}\}_{k\geq0}$. Thus, there is a subsequence $\{k_l\}_{l\geq0}$ and a vector $\bar\bmu\in\R^m$ such that 
\[\lim\limits_{l\to\infty}\frac{\sum_{t=0}^{k_l-1}{\alpha_t}\bmu^t}{\sum_{t=0}^{k_l-1}\alpha_t}=\bar\bmu.\]  
This, together with \eqref{eq:limitnabla}, yields
\begin{equation}\label{eq:left}
	\lim\limits_{l\to\infty}\frac{\sum_{t=0}^{k_l-1}\alpha_t\left(\nabla f(\x^t)+\A^{\top}\bmu^t\right)}{\sum_{t=0}^{k_l-1}\alpha_t}=\nabla f(\bar\x)+\A^{\top}\bar\bmu.
\end{equation}
We then turn to the limit of the right-hand side of \eqref{eq:tk}. We consider the interior and boundary index sets of $\bar{\x}$ separately, which are defined by
\[\cI=\left\{j:\bar{x}_j>0\right\}, \qquad\cJ=\left\{i:\bar{x}_i=0\right\},\]
respectively. We only consider the case where $\cI,\cJ\neq\emptyset$, since the proofs of other cases are nearly the same. First, note that $\{\log(\x^k_{\cI})-\log(\x^0_{\cI})\}_{k\geq0}$ is bounded due to $\x^k_{\cI}\to\bar{\x}_{\cI}>\bz$; and $\sum_{t=0}^{k-1}\alpha_t\to\infty$ by $\alpha_k\geq\underline{\alpha}>0$. We have
\begin{equation}\label{eq:right1}
	\frac1{{\sum_{t=0}^{k-1}\alpha_t}}\left(\log\left(\x^k_{\cI}\right)-\log\left(\x^0_{\cI}\right)\right)\to\bz.
\end{equation}
Second, since $\x^k_{\cJ}\to\bar{\x}_{\cJ}=\bz$ and $\x^0>\bz$, for sufficiently large $k$, we have $\x^k_{\cJ}<\x^0_{\cJ}$, or equivalently, $\log(\x^0_{\cJ})-\log(\x^k_{\cJ})>\bz$. It follows that
\begin{equation}\label{eq:right2}
	\liminf_{k\to\infty}\frac{1}{\sum_{t=0}^{k-1}\alpha_t}\left(\log\left({x^0_j}\right)-\log\left({x^k_j}\right)\right)\geq0,\quad \forall\ j\in \cJ.
\end{equation} 
Now, let $\bar\blam=\nabla f(\bar\x)+\A^{\top}\bar\bmu$ and combine \eqref{eq:tk} and \eqref{eq:left} to obtain
\[\bar\blam=\lim\limits_{l\to\infty}\frac{\sum_{t=0}^{k_l-1}\alpha_t\left(\nabla f(\x^t)+\A^{\top}\bmu^t\right)}{\sum_{t=0}^{k_l-1}\alpha_t}=\lim_{l\to\infty}\frac{\log({\x^0})-\log({\x^{k_l}})}{\sum_{t=0}^{k_l-1}\alpha_t}.\]
This, together with \eqref{eq:right1} and \eqref{eq:right2}, implies that $\bar\blam_{\cI}=\bz$ and $\bar\blam_{\cJ}\geq\bz$. It follows that $\bar\blam^{\top}\bar\x=0$. We conclude that the limit of \eqref{eq:tk} yields the vectors $\bar\bmu\in\R^m$ and $\bar\blam\in\R^n$ such that
\[\nabla f(\bar\x)+\A^{\top}\bar\bmu=\bar\blam,\qquad\bar\blam\geq0,\qquad \text{and} \qquad\bar{\blam}^{\top}\bar\x=0.   \]
By checking the definition \eqref{kkt-condition}, we see that $\bar\x$ is a critical point. This completes the proof.
\section{Proof of Proposition \ref{pro:pkl}}
\subsection{Notation and preparations}
Let $\A=[\a_1,\ldots,\a_m]^{\top}$ with $\a_i\in\R^n$. Let $\e_j$ denote the $j$-th standard basis vector of $\R^n$.
Observe that the subdifferential expression \eqref{eq:equi-con}, together with $\hat{\x}_{\cJ}=\bz$ and the positive lower bound \eqref{eq:hatxJ}, implies
\begin{equation}\label{eq:partialhat}
	\partial F(\hat\x)=\nabla f(\hat\x)+\left\{\A^{\top}\hat{\bmu}-\hat{\blam}:\hat{\bmu}\in\R^m,\hat{\blam}_{\cJ}\geq\bz,\hat{\blam}_{\cI}=\bz\right\},~\forall~\x\in\cD_P\cap\B\left(\x^*,\frac{a}{1+L_H}\right).
\end{equation}
To estimate $\dist(\bz,\partial F(\hat{\x}))$, it is helpful to clarify the structure of the set above. Note that this set is generated by the vectors $\{\a_i,\e_j:i\in[m],j\in\cJ\}$. We first present a rank identity for this collection, which will help simplify the subsequent notation.
\subsubsection{Rank identity of generating vectors}\label{sec:rank}
 We begin with the following simple observation:
 \[m+|\cJ|\geq\rank\left(\left\{\a_i,\e_j:i\in[m],j\in\cJ\right\}\right)\geq \rank\left(\left\{\e_j:j\in\cJ\right\}\right)=|\cJ|.\] 
Hence, there is an integer $\tau\in[0,m]$ such that \[\rank\left(\left\{\a_i,\e_j:i\in[m],j\in\cJ\right\}\right)=\tau+|\cJ|.\]
 By rearranging the order of the rows of $\A$, i.e., $\{\a_i:i\in[m]\}$, without loss of generality, we may assume that $\{\a_i,\e_j:i\in[\tau],j\in\cJ\}$ form a maximal linearly independent system, which implies that
\begin{equation}\label{eq:rank}
\rank\left(\left\{\a_i,\e_j:i\in[\tau],j\in \cJ\right\}\right)=\rank\left(\left\{\a_i,\e_j:i\in[m],j\in \cJ\right\}\right)=\tau+|\cJ|.
\end{equation}
The above rank identity says that the vectors $\a_{i}$, $i\in[m]\setminus[\tau]$ are linear combinations of $\a_i,\e_j$, $i\in[\tau],j\in \cJ$. Hence, we may simplify the matrix $\A$ and constraint $\A\x=\b$ via row operations.
\subsubsection{Simplification of matrix}
In the following, we focus on the case where $0<\tau<m$, since the cases where $\tau=0$ and $\tau=m$ are basically the same. Without loss of generality, we may assume that 
\[\cI=\left\{1,2,\ldots,|\cI|\right\},\qquad\cJ=\left\{|\cI|+1,\ldots,n\right\},\]
so that $\x=\left[\x_{\cI},\x_{\cJ}\right]$. With the above preparations, we are ready to simplify the matrix $\A$ via row operations.
\begin{fact}\label{fact:Q}
	There exists a non-singular matrix $\bQ\in\R^{m\times m}$ such that
		\[\bQ\A=\begin{bmatrix}
			\A_1\\
			\bz_{(m-\tau)\times|\cI|}\ \A_2 
		\end{bmatrix}
	\]
		where $\A_1=[\a_1,\ldots,\a_{\tau}]^{\top} \in \R^{\tau\times n}$ and $\A_2\in\R^{(m-\tau)\times |\cJ|}$.
\end{fact}
\begin{proof}
 The rank identity \eqref{eq:rank} implies that the linear equation
 \[\left[-\A^{\top},\e_{|\cI|+1},\ldots,\e_n\right]\v=\bz\] has $m-\tau$ linearly independent solutions, which we denote  by 
 $[\q_j,\p_j]\in\R^{m+|\cJ|}$, $j\in[m-\tau]$ with $\q_j\in\R^m$, $\p_j\in\R^{|\cJ|}$. It follows that
	\begin{equation}\label{eq:qAp}
		\A^{\top}\q_j=\left[\bz,\p_j\right].
	\end{equation}
Clearly, we have $\rank(\{\p_j:j\in [m-\tau]\})=m-\tau$. Otherwise, $\sum_{j=1}^{m-\tau}u_j\p_j=\bz$ for some $\u\in\R^{m-\tau}\setminus\{\bz\}$, which, together with \eqref{eq:qAp}, implies that $\A^{\top}\sum_{j=1}^{m-\tau}u_j\q_j=\bz$. Then, since $\A$ has full row rank, we have $\sum_{j=1}^{m-\tau}u_j\q_j=\bz$ and thus $\sum_{j=1}^{m-\tau}u_j[\q_j,\p_j]=\bz$, which contradicts the linear independence of $[\q_j,\p_j],j\in[m-\tau]$. 

	Now, let us construct $\bQ$ by
	$\bQ=[ \e_1,\ldots,\e_{\tau}, \q_1,\ldots,\q_{m-\tau}]^{\top}\in\R^{m\times m}$. By \eqref{eq:qAp}, we have  
	\begin{equation*}\label{eq:bQA}
		\bQ\A=\begin{bmatrix}
			\a_1 & \cdots& \a_{\tau} &\begin{array}{c}
				\bz\\
				\p_1
			\end{array}
			& \cdots & \begin{array}{c}
				\bz  \\
				\p_{m-\tau}
			\end{array}
		\end{bmatrix}^{\top}.
	\end{equation*}
	It remains to show that $\bQ$ is non-singular. Observe that the above expression, together with the linear independence of $\a_i,\e_j$, $i\in[\tau],j\in\cJ$ that is ensured by \eqref{eq:rank}, yields 
	\begin{align*}
		\rank(\bQ\A)=\rank\left(\left\{\a_i:i\in[\tau]\right\}\right)+\rank\left(\left\{\left[\bz,\p_j\right]:j\in [m-\tau]\right\}\right)=\tau+(m-\tau)=m,
	\end{align*}
where the second equality uses $\rank(\{\p_j:j\in [m-\tau]\})=m-\tau$.
	It follows that $\rank(\bQ)\geq \rank(\bQ\A)=m$. Hence, the matrix $\bQ\in\R^{m\times m}$ is of full rank. This completes the proof.
	\end{proof}
\subsubsection{Simplification of Constraint}
By Fact \ref{fact:Q}, the transformation matrix $\bQ$ is non-singular, so that the constraint $\A\x=\b$ in \eqref{eq:obj} is equivalent to $\bQ\A\x=\bQ\b$; i.e., 
\begin{equation}\label{eq:QA}
\begin{bmatrix}
\A_1\\
\bz_{(m-\tau)\times|\cI|}\ \A_2 
\end{bmatrix}\x=\bQ\b.
\end{equation}
Recall that $\x^*=(\x^*_{\cI},\x^*_{\cJ})$ is feasible and $\x^*_{\cJ}=\bz$. Substituting $\x^*$ into \eqref{eq:QA}, we obtain
\[\bQ\b=\left[\A_1\x^*,\A_2\x^*_{\cJ}\right]=\left[\A_1\x^*,\bz\right].\] 
Define the index sets \[\cA\coloneqq[\tau],\qquad \cB\coloneqq\{\tau+1,\ldots,m\}.\] 
Then, we have $(\bQ\b)_{\cB}=\bz$.
Given the equivalence between $\A\x=\b$ and \eqref{eq:QA},  we can assume, without loss of generality, that in Problem \eqref{eq:obj},
\begin{equation*}\label{eq:structure_A}
\A=\begin{bmatrix}
\A_1\\
\bz_{(m-\tau)\times|\cI|}\ \A_2 
\end{bmatrix}\text{ with } \A_1=[\a_1,\ldots,\a_{\tau}]^{\top},\qquad \b_{\cB}=\bz.
\end{equation*}
Under the above setting, for $\x=[\x_{\cI},\x_{\cJ}]\in\cD_P$ and $\bmu=[\bmu_{\cA},\bmu_{\cB}]\in\R^m$, we have the following formula that will prove useful for estimating $\dist(\bz,\partial F(\hat{\x}))$:
\begin{equation}\label{eq:Abmu}
	\begin{array}{rcl}
	&	\A_2\x_{\cJ}&=(\A\x)_{\cB}=\b_{\cB}=\bz, \\
	&	\A^{\top}\bmu&=\A_1^{\top}\bmu_{\cA}+\left[\bz,\A_2^{\top}\bmu_{\cB}\right], \\
	&	 \left(\A^{\top}\bmu\right)_{\cI}&=\left(\A_1^{\top}\bmu_{\cA}\right)_{\cI}, \\
 & \left(\A^{\top}\bmu\right)_{\cJ}&=\left(\A_1^{\top}\bmu_{\cA}\right)_{\cJ}+\A_2^{\top}\bmu_{\cB}.
	\end{array}
\end{equation}
\subsubsection{Implication of Strict Complementarity}
Now, let $\bmu^*\in\R^m$ and $\blam^*\in\R^n_+$ be the multipliers associated with the strict complementarity at $\x^*$, which satisfy 
\begin{equation}\label{eq:strict}
		\nabla f(\x^*)+\A^{\top}\bmu^*-\blam^*=\bz, \qquad \blam^*_{\cI}=\bz, \qquad \blam^*_{\cJ}>\bz.
\end{equation}
Define $b\coloneqq\frac12\min\{\lambda_j^*:j\in \cJ\}>0$. The following fact provides a crucial lower bound. With this in place, we are now ready to prove Proposition \ref{pro:pkl}.
\begin{fact}\label{fact:lower}
	There exists a scalar  $\epsilon_1>0$ such that for all $\x\in\R^n$, $\u\in\R^{\tau}$ satisfying
	 $\|\x-\x^*\|_2\leq\epsilon_1$ and $\|(\nabla f(\x)+\A^{\top}_1\u)_{\cI}\|_2\leq \epsilon_1$, we have
 \begin{equation*} \label{eq:propertyb}
 \left(\nabla f(\x)+\A_1^{\top}\u\right)_{\cJ}+\A^{\top}_2\bmu^*_{\cB}\geq b\1_{|\cJ|}.
 \end{equation*}
\end{fact}
\begin{proof}
Combined with $\A^{\top}\bmu^*=\A^{\top}_1\bmu^*_{\cA}+[\bz,\A_2^{\top}\bmu^*_{\cB}]$, the optimality condition \eqref{eq:strict} implies that 
\begin{equation}\label{eq:A1mu*}
	\left(\A^{\top}_1\bmu^*_{\cA}\right)_{\cI}=-\left(\nabla f(\x^*)\right)_{\cI}, \qquad \left(\nabla f(\x^*)+\A^{\top}_1\bmu^*_{\cA}\right)_{\cJ}=\blam^*_{\cJ}-\A^{\top}_2\bmu^*_{\cB}.
\end{equation}
Recall that $\A_1^{\top}=[\a_1,\ldots,\a_{\tau}]$ and $\a_i$, $\e_j$, $i\in[\tau]$, $j\in\cJ$ are linearly independent by \eqref{eq:rank}. We know that $(\A_1^{\top}\u)_{\cI}=\bz$ has a unique solution $\u=\bz$. This, together with \eqref{eq:A1mu*}, implies that $(\A^{\top}_1\u)_{\cI}=-(\nabla f(\x^*))_{\cI}$  has a unique solution $ \u=\bmu^*_{\cA}$, yielding the following rule:
\[	\left(\nabla f(\x^*)+\A^{\top}_1\u\right)_{\cI}=\bz \quad\Longrightarrow\quad \u=\bmu^*_{\cA}.  \]
Combining the above rule with the continuity of $\nabla f$, we see that
for all sequences $\{\y^k\}_{k\geq0}\subseteq\R^n$,  $\{\u^k\}_{k\geq0}\subseteq\R^{\tau}$ satisfying $\y^k\to\x^*$, $(\nabla f(\y^k)+\A^{\top}_1\u^k)_{\cI}\to\bz$, we have $\u^k\to\bmu^*_{\cA}$.

Therefore, for every $\zeta>0$, there exists an $\epsilon>0$ such that 
\begin{equation}\label{eq:zeta1}
	 \left\|\x-\x^*\right\|_2\leq\epsilon,~\left\|\left(\nabla f(\x)+\A^{\top}_1\u\right)_{\cI}\right\|_2\leq \epsilon\quad\Longrightarrow\quad \left\|\u-\bmu^*_{\cA} \right\|_2\leq\zeta.
\end{equation}
Observe that \eqref{eq:A1mu*} yields
\[\begin{aligned}
	\left\|\left(\nabla f(\x)+\A_1^{\top}\u\right)_{\cJ}+\A^{\top}_2\bmu^*_{\cB}-\blam^*_{\cJ}\right\|_2&=\left\|\left(\nabla f(\x)+\A_1^{\top}\u\right)_{\cJ}-\left(\nabla f(\x^*)+\A_1^{\top}\bmu^*_{\cA}\right)_{\cJ}\right\|_2\\
	&\leq\left\|\left(\A_1^{\top}\u-\A_1^{\top}\bmu_{\cA}^*\right)_{\cJ}\right\|_2+\left\|\left(\nabla f(\x)-\nabla f(\x^*)\right)_{\cJ}\right\|_2\\
	&\leq \|\A_1\|_2\cdot\left\|\u-\bmu^*_{\cA} \right\|_2+\left\|\nabla f(\x)-\nabla f(\x^*)\right\|_2.
\end{aligned} \]
This, together with \eqref{eq:zeta1} and the continuity of $\nabla f$, implies that for every $\zeta^{\prime}>0$, there exists an $\epsilon^{\prime}>0$ such that
\[ \left\|\x-\x^* \right\|_2\leq\epsilon^{\prime},~\left\|\left(\nabla f(\x)+\A^{\top}_1\u\right)_{\cI}\right\|_2\leq \epsilon^{\prime}\ \Longrightarrow\	\left\|\left(\nabla f(\x)+\A_1^{\top}\u\right)_{\cJ}+\A^{\top}_2\bmu^*_{\cB}-\blam^*_{\cJ}\right\|_2\leq  \zeta^{\prime}.\]
Let $\zeta^{\prime}=b$. Then, there exists an $\epsilon_1>0$ such that when $ \|\x-\x^* \|_2\leq\epsilon_1$ and $\|(\nabla f(\x)+\A^{\top}_1\u)_{\cI}\|_2\leq \epsilon_1$, we have 
\[\left\|(\nabla f(\x)+\A_1^{\top}\u)_{\cJ}+\A^{\top}_2\bmu^*_{\cB}-\blam^*_{\cJ}\right\|_2\leq b.\] 
This, together with $\lambda^*_j\geq 2b$ for $j\in\cJ$, implies that $(\nabla f(\x)+\A_1^{\top}\u)_{\cJ}+\A^{\top}_2\bmu^*_{\cB}\geq b\1_{|\cJ|}$. The proof is complete.
\end{proof}
\subsection{Main proof}
Since $\nabla f$ is locally Lipschitz continuous around $\x^*$, without loss of generality, we may assume that $\nabla f$ is $L_f$-Lipschitz continuous in the ball $\B(\x^*,\epsilon_1)$.
For notational convenience, we define the set-valued function $\cU_{\epsilon}:\R^n\rightrightarrows \R^{\tau}$ by
\[ \cU_{\epsilon}(\x)\coloneqq\left\{\u\in\R^{\tau}:\left\|\left(\nabla f(\x)+\A^{\top}_1\u\right)_{\cI}\right\|_2\leq\epsilon\right\} \]
and the scalar $\tilde{\epsilon}$ by
\[\tilde{\epsilon}\coloneqq\min\left\{\frac{\epsilon_1}{1+L_H},\frac{a}{1+L_H},\frac{\epsilon_1}{1+L_fL_H}\right\}.\]
Then, by the formula \eqref{eq:Abmu} and the definition of $\cU_{\epsilon}$, the condition $\|(\nabla f(\x)+\A^{\top}\bmu)_{\cI}\|_2\leq\tilde{\epsilon}$ can be written as $\bmu_{\cA}\in \cU_{\tilde{\epsilon}}(\x)$.

\smallskip 
\noindent\textbf{Proof of \eqref{eq:conb1}:}
Fact \ref{fact:lower}, together with the definition of $ \cU_{\epsilon}(\x)$ and $\tilde{\epsilon}<\epsilon_1$, implies that
\begin{equation*} \label{eq:blam}
	\left(\nabla f(\x)+\A^{\top}_1\bmu_{\cA}\right)_{\cJ}+\A_2^{\top}\bmu^*_{\cB}\geq b\1_{|\cJ|},\qquad\forall~\x\in\B(\x^*,\tilde{\epsilon}),\ \bmu_{\cA}\in\cU_{\tilde{\epsilon}}(\x).
\end{equation*}
Let $\v(\x,\bmu)\coloneqq(\nabla f(\x)+\A^{\top}_1\bmu_{\cA})_{\cJ}+\A_2^{\top}\bmu^*_{\cB}-b\1_{|\cJ|}$. The above inequality ensures that
\begin{equation}\label{eq:blamv}
 \v(\x,\bmu)\geq \bz,\qquad \forall~\x\in\B(\x^*,\tilde{\epsilon}),\ \bmu_{\cA}\in\cU_{\tilde{\epsilon}}(\x).
\end{equation}
Using the definition of $\v(\x,\bmu)$ and the formulas in \eqref{eq:Abmu}, we have
\[\begin{aligned}
	&\quad\left\|\Diag\left(\sqrt{\x_{\cJ}}\right)\left(\nabla f(\x)+\A^{\top}\bmu\right)_{\cJ}\right\|_2^2\\
	&=\left(\left(\nabla f(\x)+\A_1^{\top}\bmu_{\cA}\right)_{\cJ}+\A_2^{\top}\bmu_{\cB}\right)^{\top}\Diag(\x_{\cJ})\left(\left(\nabla f(\x)+\A_1^{\top}\bmu_{\cA}\right)_{\cJ}+\A_2^{\top}\bmu_{\cB}\right)\\
	&=\left(b\1_{|\cJ|}+\v(\x,\bmu)+\A_2^{\top}\left(\bmu_{\cB}-\bmu^*_{\cB}\right)\right)^{\top}\Diag(\x_{\cJ})\left(b\1_{|\cJ|}+\v(\x,\bmu)+\A_2^{\top}\left(\bmu_{\cB}-\bmu^*_{\cB}\right)\right) \\
 &=b^2\cdot\1_{|\cJ|}^{\top}\Diag(\x_{\cJ})\1_{|\cJ|}+2b\cdot\1_{|\cJ|}^{\top}\Diag(\x_{\cJ})\left(\v(\x,\bmu)+\A_2^{\top}\left(\bmu_{\cB}-\bmu^*_{\cB}\right)\right) \\
 &\quad \  +\left(\v(\x,\bmu)+\A_2^{\top}\left(\bmu_{\cB}-\bmu^*_{\cB}\right)\right)^{\top}\Diag(\x_{\cJ})\left(\v(\x,\bmu)+\A_2^{\top}\left(\bmu_{\cB}-\bmu^*_{\cB}\right)\right).   
\end{aligned} 
\]
Observe that (i) $\1_{|\cJ|}^{\top}\Diag(\x_{\cJ})=\x_{\cJ}^{\top}$ ; (ii) the first term on the right-hand side satisfies \[\1_{|\cJ|}^{\top}\Diag(\x_{\cJ})\1_{|\cJ|}=\x_{\cJ}^{\top}\1_{|\cJ|}=\left\|\x_{\cJ}\right\|_1\geq \left\|\x_{\cJ}\right\|_2; \]
and (iii) the last term is nonnegative due to  $\Diag(\x_{\cJ})\succeq\bz$. It follows that
\[ \begin{aligned}
&\left\|\Diag\left(\sqrt{\x_{\cJ}}\right)\left(\nabla f(\x)+\A^{\top}\bmu\right)_{\cJ}\right\|_2^2	\\
\geq&\ b^2\cdot\left\|\x_{\cJ}\right\|_2+2b\cdot\x_{\cJ}^{\top}\left(\v(\x,\bmu)+\A_2^{\top}\left(\bmu_{\cB}-\bmu^*_{\cB}\right)\right)\\
=&\ b^2 \cdot\left\|\x_{\cJ}\right\|_2+2b\cdot\x^{\top}_{\cJ}\v(\x,\bmu)+2b\cdot\x^{\top}_{\cJ}\A_2^{\top}\left(\bmu_{\cB}-\bmu^*_{\cB}\right). 
\end{aligned} 
\] 
Note that (i) the inequality \eqref{eq:blamv} and $\x_{\cJ}\geq\bz$ ensure that $\x^{\top}_{\cJ}\v(\x,\bmu)\geq 0$ for $\x\in\B(\x^*,\tilde{\epsilon})$, $ \bmu_{\cA}\in\cU_{\tilde{\epsilon}}(\x)$; and (ii) $\A_2\x_{\cJ}=\bz$ by \eqref{eq:Abmu}. The above inequality further implies that
\[\left\|\Diag\left(\sqrt{\x_{\cJ}}\right)\left(\nabla f(\x)+\A^{\top}\bmu\right)_{\cJ}\right\|_2^2\geq b^2\left\|\x_{\cJ}\right\|_2,\quad\forall~\x\in\B(\x^*,\tilde{\epsilon}),~  \bmu_{\cA}\in\cU_{\tilde{\epsilon}}(\x), \]
which establishes \eqref{eq:conb1} after taking a square root. 

\smallskip
\noindent\textbf{Proof of \eqref{eq:disthatx}:}
The subdifferential expression \eqref{eq:partialhat}, along with $\tilde{\epsilon}\leq a/(1+L_H)$ and the formulas in \eqref{eq:Abmu}, ensures that for $\x\in\cD_P\cap\B(\x^*,\tilde{\epsilon})$,
\[\partial F(\hat{\x})=\left\{\nabla f(\hat{\x})+\bm{A}^{\top}_1\hat{\bmu}_{\cA}+\left[\bz,\A^{\top}_2\hat{\bmu}_{\cB}\right]-\hat{\blam}:\hat{\bmu}\in\R^m,\hat{\blam}_{\cJ}\geq\bz,\hat{\blam}_{\cI}=\bz\right\}.\]
Based on the above expression, we have an alternative formulation of $\dist(\bz,\partial F(\hat{\x}))$:
\begin{equation}\label{eq:alternative}
	\dist\left(\bz,\partial F(\hat{\x})\right)=\min_{\hat{\bmu}_{\cA}\in\R^{\tau}}\min_{\substack{\hat{\blam}_{\cJ}\geq\bz,\\
			\hat{\bmu}_{\cB}\in\R^{m-\tau} } }~ \left\| \nabla f(\hat{\x})+\bm{A}^{\top}_1\hat{\bmu}_{\cA}+\left[\bz,\A^{\top}_2\hat{\bmu}_{\cB}\right]-\left[\bz,\hat{\blam}_{\cJ}\right]\right\|_2.
\end{equation}
Observe that the inner minimization problem can be simplified as follows:
\begin{equation}\label{eq:min1}
	\begin{aligned}
		&\min_{\substack{\hat{\blam}_{\cJ}\geq\bz,\\
				\hat{\bmu}_{\cB}\in\R^{m-\tau} } } ~\left\|\nabla f(\hat{\x})+\bm{A}^{\top}_1\hat{\bmu}_{\cA}+[\bz,\A^{\top}_2\hat{\bmu}_{\cB}]-\left[\bz,\hat{\blam}_{\cJ}\right]\right\|_2 \\
		=&\left\|(\nabla f(\hat{\x})+\bm{A}^{\top}_1\hat{\bmu}_{\cA})_{\cI}\right\|_2+\min_{\substack{\hat{\blam}_{\cJ}\geq\bz,\\
				\hat{\bmu}_{\cB}\in\R^{m-\tau} } } ~ \left\|\left(\nabla f(\hat{\x})+\bm{A}^{\top}_1\hat{\bmu}_{\cA}\right)_{\cJ}+\A_2^{\top}\hat{\bmu}_{\cB}-\hat{\blam}_{\cJ}\right\|_2.\\
	\end{aligned}
\end{equation}
Note that the inequality \eqref{eq:hatxx*} and $\tilde{\epsilon}\leq {\epsilon_1}/(1+L_H)$ imply that $\hat{\x}\in\B(\x^*,\epsilon_1)$  for $\x\in\cD_P\cap\B(\x^*,\tilde{\epsilon})$. This, together with Fact \ref{fact:lower}, yields 
\begin{equation*}\label{eq:hatblam}
	\left(\nabla f(\hat{\x})+\bm{A}_1^{\top}\hat{\bmu}_{\cA}\right)_{\cJ}+\A_2^{\top}\bmu^*_{\cB}\geq b\1_{|\cJ|}>\bz
\end{equation*}
for $\x\in\cD_P\cap\B(\x^*,\tilde{\epsilon})$ and $\hat\bmu_{\cA}\in\cU_{\epsilon_1}(\hat{\x})$. This positive lower bound ensures that $\hat{\bmu}_{\cB}=\bmu^*_{\cB}$ and  $\hat{\blam}_{\cJ}=(\nabla f(\hat{\x})+\A^{\top}_1\hat{\bmu}_{\cA})_{\cJ}+\A^{\top}_2\bmu^*_{\cB}>\bz$ are optimal for the minimization problem on the right-hand side of \eqref{eq:min1}, and the optimal value is zero. 
It follows that for $\x\in\cD_P\cap\B(\x^*,\tilde{\epsilon})$ and $\hat{\bmu}_{\cA}\in\cU_{\epsilon_1}(\hat{\x})$,
\[\min_{\substack{\hat{\blam}_{\cJ}\geq\bz,\\
		\hat{\bmu}_{\cB}\in\R^{m-\tau} } } ~ \left\|\nabla f(\hat{\x})+\bm{A}^{\top}_1\hat{\bmu}_{\cA}+\left[\bz,\A^{\top}_2\hat{\bmu}_{\cB}\right]-\left[\bz,\hat{\blam}_{\cJ}\right]\right\|_2 \\
=\left\|\left(\nabla f(\hat{\x})+\bm{A}^{\top}_1\hat{\bmu}_{\cA}\right)_{\cI}\right\|_2.\]
Then, when $\cU_{\epsilon_1}(\hat{\x})\neq\emptyset$, combining the above equation and the formulation \eqref{eq:alternative}, we see that for $\x\in\cD_P\cap\B(\x^*,\tilde{\epsilon})$,
\begin{equation}\label{eq:keyupper}
	\begin{aligned}
			\dist\left(\bz,\partial F(\hat{\x})\right)&\leq\min_{\hat{\bmu}_{\cA}\in\cU_{\epsilon_1}(\hat{\x})}\min_{\substack{\hat{\blam}_{\cJ}\geq\bz,\\
					\hat{\bmu}_{\cB}\in\R^{m-\tau} } }~\left\| \nabla f(\hat{\x})+\bm{A}^{\top}_1\hat{\bmu}_{\cA}+\left[\bz,\A^{\top}_2\hat{\bmu}_{\cB}\right]-\left[\bz,\hat{\blam}_{\cJ}\right]\right\|_2\\
				&= \min_{\hat{\bmu}_{\cA}\in\cU_{\epsilon_1}(\hat{\x})}~ \left\|\left(\nabla f(\hat{\x})+\bm{A}^{\top}_1\hat{\bmu}_{\cA}\right)_{\cI}\right\|_2.
	\end{aligned}
\end{equation}
The condition $\cU_{\epsilon_1}(\hat{\x})\neq\emptyset$ is key to the above inequality. Let us show that it holds under the setting $\x\in\cD_P\cap\B(\x^*,\tilde{\epsilon})$ and $\bmu_{\cA}\in\cU_{\tilde{\epsilon}}(\x)$ (recall that this is a shorthand for $\|(\nabla f(\x)+\A^{\top}\bmu)_{\cI}\|_2\leq\tilde{\epsilon}$). Note that by $\tilde{\epsilon}\leq {\epsilon_1}/(1+L_H)$ and \eqref{eq:hatxx*},  we have  \[\x,\hat{\x}\in\cD_P\cap\B(\x^*,\epsilon_1),\qquad\forall~\x\in\cD_P\cap\B(\x^*,\tilde{\epsilon}).\] 
Since $\nabla f$ is $L_f$-Lipschitz continuous in $\B(\x^*,\epsilon_1)$, we have
\begin{equation}\label{eq:nablaf_f}
	\left\|\nabla f(\x)-\nabla f(\hat{\x})\right\|_2\leq L_f \left\|\hat{\x}-\x\right\|_2\leq L_fL_H\left\|\x_{\cJ}\right\|_2, \qquad\forall~\x\in\cD_P\cap\B(\x^*,\tilde{\epsilon}),
\end{equation}
where the second inequality uses the error bound \eqref{eq:hatxJ1}.
Since $\|\x_{\cJ}\|_2\leq\|\x-\x^*\|_2$ by $\x^*_{\cJ}=\bz$, the inequality \eqref{eq:nablaf_f} further yields
\[\left\|\nabla f(\x)-\nabla f(\hat{\x})\right\|_2 \leq L_fL_H\left\|\x-\x^*\right\|_2\leq L_fL_H\tilde{\epsilon},\qquad\forall~\x\in\cD_P\cap\B(\x^*,\tilde{\epsilon}).\]
This, together with the choice $\tilde{\epsilon}\leq\epsilon_1/(1+L_fL_H)$, implies an upper bound on the norm $\|(\nabla f(\hat{\x})+\A^{\top}_1\u)_{\cI}\|_2$ for $\x\in\cD_P\cap\B(\x^*,\tilde{\epsilon})$, $\u\in \cU_{\tilde{\epsilon}}(\x)$:
\[\begin{aligned}
	\left\|\left(\nabla f(\hat{\x})+\A^{\top}_1\u\right)_{\cI}\right\|_2\leq&\left\|(\nabla f({\x})+\A^{\top}_1\u)_{\cI}\right\|_2+ \left\|\left(\nabla f(\hat{\x})-\nabla f(\x)\right)_{\cI}\right\|_2\\
	\leq &~\tilde{\epsilon}+\left\|\nabla f(\hat{\x})-\nabla f(\x)\right\|_2  \\
	\leq & ~(1+L_fL_H)\tilde{\epsilon} \\
	\leq&~\epsilon_1.
\end{aligned}\]
The above bound says that $\u\in \cU_{\epsilon_1}(\hat{\x})$  if $\u\in \cU_{\tilde{\epsilon}}(\x)$ and  $\x\in\cD_P\cap\B(\x^*,\tilde{\epsilon})$. Hence, under the conditions $\x\in\cD_P\cap\B(\x^*,\tilde{\epsilon})$ and $\bmu_{\cA}\in\cU_{\tilde{\epsilon}}(\x)$, we have
\[\bmu_{\cA}\in  \cU_{\epsilon_1}(\hat{\x}).\]
The above inclusion ensures that $\cU_{\epsilon_1}(\hat{\x})\neq\emptyset$. Moreover,
together with \eqref{eq:keyupper}, it implies that
\[	\dist\left(\bz,\partial F(\hat{\x})\right)\leq\left\|(\nabla f(\hat{\x})+\bm{A}^{\top}_1{\bmu}_{\cA})_{\cI}\right\|_2,\qquad\forall~\x\in\cD_P\cap\B(\x^*,\tilde{\epsilon}),~ \bmu_{\cA}\in\cU_{\tilde{\epsilon}}(\x) .\]
Combining this with the inequality \eqref{eq:nablaf_f} yields
\[ \begin{aligned}
		\dist\left(\bz,\partial F(\hat{\x})\right)&\leq\left\|(\nabla f(\x)+\bm{A}^{\top}_1{\bmu}_{\cA})_{\cI}\right\|_2+	\left\|\nabla f(\x)-\nabla f(\hat{\x})\right\|_2\\
		&\leq \left\|(\nabla f(\x)+\bm{A}^{\top}_1{\bmu}_{\cA})_{\cI}\right\|_2+ L_fL_H\left\|\x_{\cJ}\right\|_2
\end{aligned}\]
for $\x\in\cD_P\cap\B(\x^*,\tilde{\epsilon})$ and $\bmu_{\cA}\in\cU_{\tilde{\epsilon}}(\x)$. This establishes \eqref{eq:disthatx}. 

\section*{Acknowledgements}
We would like to thank Dr. Jinxin Wang and Dr. Linglingzhi Zhu for their comments on earlier versions of this manuscript, which help to improve its presentation. 

We did not use AI to develop the ideas or proofs in this paper. AI tools (primarily GPT-4o) were used solely for checking grammar and typos.

\bibliographystyle{plainnat}
\bibliography{ref}
\end{document}